\documentclass[reqno,12pt]{amsart}
\usepackage{amsfonts}
\usepackage{graphicx}
\usepackage{amsfonts,amsmath, amssymb}
\usepackage{marginnote}
\usepackage{color}

\usepackage[margin=1.1in]{geometry} 
\usepackage[colorlinks,linkcolor=blue,anchorcolor=red,citecolor=blue]{hyperref}
\numberwithin{equation}{section}

\allowdisplaybreaks

\newcommand{\R}{\mathbb{R}}

\newtheorem{theorem}{Theorem}[section]
\newtheorem{corollary}[theorem]{Corollary}
\newtheorem{lemma}[theorem]{Lemma}
\newtheorem{proposition}[theorem]{Proposition}
\newtheorem{remark}[theorem]{Remark}
\newtheorem{definition}[theorem]{Definition}

\def\i{\infty}
\def\f{\frac}

\def\b{\bar}

\usepackage{tikz}

\newcommand{\dd}{{\rm d}}

\newcommand{\mG}{\mathbf{\Gamma}}

\newcommand{\Fi}{\mathbf{1}}

\newcommand{\ga}{\gamma}

\newcommand{\pa}{\partial}
\newcommand{\ka}{\kappa}
\newcommand{\eps}{\epsilon}

\newcommand{\vep}{\varepsilon}

\begin{document}

\title[Shock profiles for cutoff Boltzmann of a binary gas mixture]{Shock profiles for the cutoff Boltzmann equation of a binary gas mixture}
\author[R.-J. Duan]{Renjun Duan}
\address[RJD]{Department of Mathematics, The Chinese University of Hong Kong, Shatin, Hong Kong, P.R.~China}
\email{rjduan@math.cuhk.edu.hk}

\author[Z.-G. Li]{Zongguang Li}
\address[ZGL]{Department of Applied Mathematics, The Hong Kong Polytechnic University, Hung Hom, Hong Kong, P.R.~China}
\email{zongguang.li@polyu.edu.hk}

\author[Z. Zhang]{Zhu Zhang}
\address[ZZ]{Department of Applied Mathematics, The Hong Kong Polytechnic University, Hung Hom, Hong Kong, P.R.~China}
\email{zhuama.zhang@polyu.edu.hk}

\begin{abstract}
We prove the existence of small-amplitude traveling shock profiles for the one-dimensional Boltzmann equation of a binary gas mixture with angular cutoff potentials in the full range $-3<\gamma\le 1$. The result extends the classical construction of Caflisch and Nicolaenko from hard potentials to the cutoff soft-potential regime. Indeed, the argument of proofs combines a Lyapunov--Schmidt reduction of the macroscopic component to a Burgers equation with an accelerated backward bi-characteristic method and a weighted $L^2$--$L^\infty$ iteration. Acceleration restores a uniformly positive collision frequency, compensating for the lack of a spectral gap for soft potentials, while the $L^2$--$L^\infty$ framework accommodates the absence of velocity smoothing induced by the cutoff, including a possible singularity along the grazing characteristic $v_1=s$. The shock profile tends to the Rankine--Hugoniot bi-Maxwellians at a mixed exponential rate as $|x|\to \infty$, with a sub-exponential remainder of order $|\varepsilon x|^{2/(3-\gamma)}$.
\end{abstract}

\date{\today}

\subjclass[2000]{35Q20, 35L67, 35C07, 35A01}


\keywords{Boltzmann equation, shock waves, cutoff soft potentials, existence}

\maketitle


\thispagestyle{empty}
\section{Introduction}
\subsection{Boltzmann equation for a binary gas mixture}
Under the assumption of the slab symmetry in space, the motion of a binary gas mixture is governed by the following one-dimensional two-species Boltzmann equations (cf.~\cite{ABT,KAT}):
\begin{equation}\label{1.1.1}
\left\{\begin{aligned}
&\pa_tF_A+v_1\pa_xF_A=Q^{AA}(F_A,F_A)+Q^{BA}(F_B,F_A),\\
&\pa_tF_B+v_1\pa_xF_B=Q^{AB}(F_A,F_B)+Q^{BB}(F_B,F_B),
\end{aligned}\right.
\end{equation}
or in the vector form
\begin{align}\label{1.1.1vf}
\pa_t\mathbf{F}+v_1\pa_x\mathbf{F}=\mathbf{Q(F)},
\end{align}
with $ \mathbf{F}=\mathbf{F}(t,x,v):=[F_A,F_B]^T$ and 
\begin{align}\label{def.qfv}
\mathbf{Q(F)}:=[\sum_{j=A,B}Q^{jA}(F_j,F_A),\sum_{j=A,B}Q^{jB}(F_j,F_B)]^T.
\end{align}
The unknown $F_{A}=F_A(t,x,v)$ is the velocity distribution function of the $A$-species gas particles and $F_B$ is that of the $B$-species gas particles, which have position $x\in \mathbb{R}$ and velocity $v=(v_1,v_2,v_3)\in \mathbb{R}^3$ at time $t\in \R$. The collision operator $Q^{ji}$ is defined by
\begin{align}\notag
Q^{ji}(F_{j},G_{i}):=\int_{\mathbb{R}^3\times \mathbb{S}^2}B^{ji}(|v-u|,\cos\theta)[F_{j}(u')G_{i}(v')-F_{j}(u)G_{i}(v)]\dd u\dd \sigma,
\end{align}
for $i,j\in \{A,B\}$. Here $v'=(v')^{ji}$ and $u'=(u')^{ji}$ are the velocities of $i$ and $j$- species after collision, and $\cos\theta:= \sigma\cdot (v-u)/|v-u|$. Let $m_i$ and $m_j$ be denoted as the mass of molecule of $i$ and $j$ species, respectively,  and 
$$
V:=\f{m_{i}v+m_j u}{m_i+m_{j}}
$$ 
as the center of momentum, then the post-collision velocities are computed as
\begin{align}\label{1.1.3}
v'=V+\f{m_{j}}{m_{i}+m_{j}}|v-u|\sigma,\quad u'=V-\f{m_{i}}{m_{i}+m_{j}}|u-v|\sigma,
\end{align}
in terms of the conservation of molecular momentum and energy that
\begin{align}\notag
m_i v+m_{j}u=m_{i}v'+m_{j}u',\quad m_i |v|^2+m_{j}|u|^2=m_{i}|v'|^2+m_{j}|u'|^2.
\end{align}
The collision kernel $B^{ji}$ is non-negative and takes the form of
$$B^{ji}(|v-u|,\cos\theta)=|v-u|^{\gamma}b^{ji}(\theta),\quad i,j\in \{A,B\},
$$
where $-3<\gamma\leq 1$ and $ b^{ji}(\theta)$ satisfies the Grad's angular cut-off assumption that
$$
0\leq b^{ji}(\theta)\leq C|\cos\theta\sin\theta|.
$$
It is well-known (eg.~\cite{ABT}) that corresponding to the global equilibrium of the system \eqref{1.1.1vf} or \eqref{1.1.1}, the following form 
\begin{align}\label{1.1.4.1}
\textbf{M}=[M_A,M_B]^T=\left[\f{\rho_Am^{3/2}_A}{(2\pi\theta)^{3/2}}
e^{-\f{m_A|v-u|^2}{2\theta}},\f{\rho_Bm^{3/2}_B}{(2\pi\theta)^{3/2}}e^{-\f{m_B|v-u|^2}{2\theta}}\right]^T
\end{align}
is called a bi-Maxwellian whose two components of the Maxwellian form have the same bulk velocities and temperatures. Here, the constants $\rho_{A},\rho_B>0$ represent respectively the density of $A$ and $B$-species, $u=(u_1,u_2,u_3)\in \mathbb{R}^3$ is the bulk velocity, and $\theta>0$ denotes the temperature. Note that the momentum or heat transfer between different species turns out to vanish when the gas mixture reaches into the equilibrium, cf.~\cite{ABT}. 

\subsection{Shock profile}
The shock profile for \eqref{1.1.1} is a planar travelling wave solution in the form of $\mathbf{F}=\widetilde{\mathbf{F}}(x-st,v):=[\widetilde{F}_A,\widetilde{F}_B]^T(x-st,v)$ with the shock speed $s$ to be specified later. We substitute this form into \eqref{1.1.1} and drop the tildes to obtain the following steady problem:
\begin{equation}\label{1.2.1}\left\{
\begin{aligned}
&(v_1-s)\pa_{x}F_A=Q^{AA}(F_A,F_A)+Q^{BA}(F_B,F_A),\\
&(v_1-s)\pa_{x}F_B=Q^{AB}(F_A,F_B)+Q^{BB}(F_B,F_B),
\end{aligned}\right.
\end{equation}
with the far-field conditions:
\begin{align}\label{1.2.2}
\lim_{x\rightarrow\pm\infty}[F_A,F_B]^T(x,v)=\textbf{M}_{\pm}(v)=[M_{A,\pm},M_{B,\pm}]^T(v),
\end{align}
where
\begin{equation}\label{def.ffdata.pm}
M_{i,\pm}=M_{i,\pm}(v):=\f{\rho_{i,\pm}m^{3/2}_i}{(2\pi\theta_{\pm})^{3/2}}
e^{-\f{m_i(|v_1-u_{\pm}|^2+|v_2|^2+|v_3|^2)}{2\theta_{\pm}}},\ i=A,B.
\end{equation} 
Note that  in terms of \eqref{def.ffdata.pm}, both the far-field data $\textbf{M}_{+}(v)$ and $\textbf{M}_{-}(v)$ at $x=\pm\infty$ in \eqref{1.2.2} are bi-Maxwellians. 

To derive the Rankine--Hugoniot conditions for the far-field  macroscopic fluid quantities $[\rho_{A,\pm},\rho_{B,\pm},u_{\pm},\theta_{\pm}]$ of \eqref{def.ffdata.pm} associated to the far-field conditions \eqref{1.2.2}, one should first investigate the local conservation laws from \eqref{1.1.1}. Thus, we define the fluid variables $\rho_{\alpha},u_{\alpha}=[u_{\alpha,1},u_{\alpha,2},u_{\alpha,3}]$ and $\theta_{\alpha}$ associated to the distribution $F_{\alpha}$ through the following moments:
\begin{equation}
\left\{\begin{aligned}
&\rho_{\alpha}(t,x)=\int_{\mathbb{R}^3}F_{\alpha}(t,x,v)\dd v,\nonumber\\ 
&\rho_{\alpha}u_{\alpha,i}(t,x)=\int_{\mathbb{R}^3}v_iF_{\alpha}(t,x,v)\dd v,\nonumber\\
&\f{3}{2}\rho_{\alpha}\theta_{\alpha}(t,x)+\f{1}{2}m_\alpha\rho_{\alpha}|u_{\alpha}|^2(t,x)
=\int_{\mathbb{R}^3}\f{m_{\alpha}|v|^2}{2}F_{\alpha}(t,x,v)\dd v,
\end{aligned}\right.
\quad\alpha\in\{A,B\}.
\end{equation}
It is well-known (eg.~\cite{ABT}) that the collision term $\mathbf{Q}(\mathbf{F})$ in \eqref{def.qfv} has the following six collision invariants:
\begin{align*}
\psi_{-1}=[1,0]^T,\ \psi_0=[0,1]^T;\ \psi_i=[m_Av_i,m_Bv_i]^T,\ i=1,2,3;\ \psi_4=[\f{m_{A}|v|^2}{2},\f{m_B|v|^2}{2}]^T,
\end{align*}
with
\begin{align}\label{1.2.2.2}
\int_{\mathbb{R}^3}\psi_j\cdot \textbf{Q}(\mathbf{F})\dd v=0,\quad j=-1,0,1,2,3,4.
\end{align}
For the equation \eqref{1.1.1vf}, as in \cite{Liu-Yang-Yu},  we make the decomposition that $\textbf{F}=\textbf{M}+\textbf{G}$ with 
\begin{align}\label{def.lm.ab}
\textbf{M}:=[\f{\rho_{A}m^{3/2}_A}{(2\pi\theta_{A})^{3/2}}
e^{-\f{m_A(|v_1-u_{A}|^2+|v_2|^2+|v_3|^2)}{2\theta_{A}}},\f{\rho_{B}m^{3/2}_B}{(2\pi\theta_{B})^{3/2}}
e^{-\f{m_B(|v_1-u_{B}|^2+|v_2|^2+|v_3|^2)}{2\theta_{B}}}]^T,
\end{align}
such that 
$$
\int_{\mathbb{R}^3}\psi_j\cdot \textbf{G}\dd v=0, \quad j=-1,\cdots, 4.
$$ 
Note that $\textbf{M}$ in \eqref{def.lm.ab} is not a bi-Maxwellian.   
Then, the fact that $\int_{\mathbb{R}^3}\psi_j\cdot\{\pa_t+v_1\pa_x \}\textbf{F}\dd v=\int_{\mathbb{R}^3}\psi_j\cdot \textbf{Q(F)}\dd v=0$ implies the following hydrodynamical-type system for the fluid variables $[\rho_{\alpha},u_{\alpha},\theta_{\alpha}]$ ($\alpha=A,B)$ with the coupling to $\textbf{G}$:
\begin{equation}\label{1.2.3}
\left\{\begin{aligned}
&\pa_t\rho_A+\pa_x(\rho_Au_{A,1})=0,\quad\pa_t\rho_B+\pa_x(\rho_Bu_{B,1})=0,\\
&\pa_t[m_{A}\rho_Au_{A,1}+m_{B}\rho_Bu_{B,1}]+\pa_x[m_A\rho_Au_{A,1}^2+m_B\rho_Bu_{B,1}^2
]+\pa_xP=-\int_{\mathbb{R}^3}v_1\pa_x \textbf{G}\cdot\psi_1\dd v,\\
&\pa_t[m_{A}\rho_Au_{A,i}+m_{B}\rho_Bu_{B,i}]+\pa_x[m_A\rho_Au_{A,1}u_{A,i}+m_B\rho_Bu_{B,1}u_{B,i}]\\
&\qquad\qquad\qquad\qquad\qquad\qquad\qquad=-\int_{\mathbb{R}^3}v_1\pa_x \textbf{G}\cdot\psi_i\dd v,\quad i=2,3,\\
&\pa_t[\mathcal{E}+\f{1}{2}(m_A\rho_A|u_A|^2+m_B\rho_B|u_B|^2)]+\pa_x[\f{1}{2}(m_A\rho_A|u_A|^2u_{A,1}+m_B\rho_B|u_B|^2u_{B,1})]\\
&\qquad\qquad\qquad\qquad+\pa_x[
\f{5}{2}(\rho_A\theta_Au_{A,1}+\rho_B\theta_Bu_{B,1})]=-\int_{\mathbb{R}^3}
v_1\pa_x\textbf{G}\cdot\psi_4\dd v,
\end{aligned}\right.
\end{equation}
where $P=\f23\mathcal{E}=\rho_A\theta_A+\rho_B\theta_B$. Note that the fluid-type system \eqref{1.2.3} is not closed since $\mathbf{G}$ depends on the higher-order moments of $\textbf{F}$. By substituting the ansatz $\mathbf{F}=\widetilde{\mathbf{F}}(x-st,v)$ into \eqref{1.2.3}, dropping the tildes and integrating from $-\infty$ to $+\infty$, we get the following closed Rankine--Hugoniot conditions:
\begin{equation}\label{1.2.4}
\left\{\begin{aligned}
&-s(\rho_{A,+}-\rho_{A,-})+\rho_{A,+}u_{+}-\rho_{A,-}u_-=0,\\
&-s(\rho_{B,+}-\rho_{B,-})+\rho_{B,+}u_{+}-\rho_{B,-}u_-=0,\\
&-s(\bar{m}_+u_+-\bar{m}_-u_-)+\bar{m}_+u_+^2-\bar{m}_-u_-^2+\b{P}_{+}-\b{P}_{-}=0,\\
&-s(\b{e}_+-\b{e}_-)+\f{5}{2}\b{\rho}_+\theta_+u_+
-\f52\b{\rho}_-\theta_-u_-+\f12\b{m}_+u_+^3-\f12\b{m}_-u_-^3=0,
\end{aligned}\right.
\end{equation}
where
\begin{align}\label{1.2.5}
\b{m}_{\pm}=m_{A}\rho_{A,\pm}+m_{B}\rho_{B,\pm},\ \b{\rho}_{\pm}=\rho_{A,\pm}+\rho_{B,\pm},\ \b{P}_{\pm}=\b{\rho}_{\pm}\theta_{\pm},\ \b{e}_{\pm}=\f{3}{2}\b{\rho}_{\pm}\theta_{\pm}+\f12\b{m}_{\pm}u_{\pm}^2.
\end{align}
Without loss of generality, we assume that $u_-=0$ and $\theta_-=1$ in the sequel. Denote the parameter 
\begin{align}\label{def.etas}
\eta=\f{\rho_{A,-}}{\rho_{A,+}}-1,
\end{align}
which measures the strength of shock profile. Then a direct calculation from \eqref{1.2.4} implies that the speed of shock is given by
\begin{align}\notag
s=\pm\sqrt{\f{5\b{\rho}_-}{3\b{m}_-+4\b{m}_-\eta}},
\end{align}
and the admissible macroscopic quantities $[\rho_{A,+},\rho_{B,+},u_+,\theta_+]$ of the downstream state
composes the following one-dimensional Hogoniot curve parameterized by $\eta$ in \eqref{def.etas} as
\begin{align}
&\rho_{B,-}/\rho_{B,+}=\rho_{A,-}/\rho_{A,+}=1+\eta,\nonumber\\
&u_+=-s\eta,\quad \theta_+-1=\eta-\f{5s^2}{3c_0^2}\eta(1+\eta),\notag
\end{align}
with $c_0$ being the sound speed. In this paper, we only consider the shock profile with positive speed, that is, 
$$
c_0=\sqrt{\f{5\b{\rho}_-}{3\b{m}_-}},\quad 
s=\sqrt{\f{5\b{\rho}_-}{3\b{m}_-+4\b{m}_-\eta}}.
$$
We also assume the shock profile is compressive, in other words, the following entropy condition holds:
\begin{align}\label{1.2.7-1}
c_0>s.
\end{align}
\begin{remark}
From the kinetic point of view, \eqref{1.2.7-1} can be derived from the following entropy inequality
\begin{align}
\pa_x\int_{\mathbb{R}^3}(v_1-s)\mathbf{F}\cdot\log\mathbf{F}\dd v=\int_{\mathbb{R}^3}\mathbf{Q(F)}\cdot\log\mathbf{F}\dd v\leq 0,\nonumber
\end{align}
when the shock wave is weak.
\end{remark}

Fixing the upstream state $\mathbf{M}_-=[M_{A,-},M_{B,-}]^T$, we look for the shock profile solution $\mathbf{F}=[F_A,F_B]^T$ in the form of
\begin{equation}\label{perturb}
F_{A}=M_{A,-}+M_{A,-}^{1/2}f_A,\quad F_{B}=M_{B,-}+M_{B,-}^{1/2}f_B.
\end{equation}
For any bi-Maxwellian $\textbf{M}=[M_A,M_B]^T,$ the linearized collision operator $\mathbf{L}_\textbf{M}$ is defined by
\begin{equation}\label{1.2.10}
	\mathbf{L}_\textbf{M}\mathbf{f}=\mathbf{L}_\textbf{M}\mathbf{f}(x,v)=
	\left[\begin{aligned}
		&-\f1{M_{A}^{1/2}}\sum_{j=A,B}\{Q^{j A}(M_{j},M_{A}^{1/2}f_A)+Q^{j A}(M_{j}^{1/2}f_j,M_{A})\}\\
		&-\f1{M_{B}^{1/2}}\sum_{j=A,B}\{Q^{j B}(M_{j},M_{B}^{1/2}f_B)+Q^{j B}(M_{j}^{1/2}f_j,M_{B})\}
	\end{aligned}
	\right].
\end{equation} 
Substituting \eqref{perturb} into \eqref{1.2.1}, we see that $\mathbf{f}:=[f_A,f_B]^T=[f_A,f_B]^T(x,v)$ solves
\begin{align}\label{1.2.8}
(v_1-s)\pa_x\mathbf{f}+\mathbf{Lf}=\mathbf{\Gamma}(\mathbf{f},\mathbf{f}),
\end{align}
with two far fields
\begin{align}\label{1.2.9}
\lim_{x\rightarrow-\infty}[f_A,f_B]^T(x,v)=[0,0]^T,\quad\lim_{x\rightarrow +\infty}[f_A,f_B]^T(x,v)=[\f{M_{A,+}-M_{A,-}}{\sqrt{M_{A,-}}},\f{M_{B,+}-M_{B,-}}{\sqrt{M_{B,-}}}]^T,
\end{align}
where we have denoted $\mathbf{L}=\mathbf{L}_{\textbf{M}_-}$ 
and
\begin{equation}\notag
\mathbf{\Gamma}(\mathbf{f},\mathbf{g})=
\left[\begin{aligned}
&\sum_{j=A,B}\f1{M_{A,-}^{1/2}}Q^{j A}(M_{j,-}^{1/2}f_{j},M_{A,-}^{1/2}g_A)\\
&\sum_{j=A,B}\f1{M_{B,-}^{1/2}}Q^{j B}(M_{j,-}^{1/2}f_{j},M_{B,-}^{1/2}g_B)
\end{aligned}
\right]:
=\left[\begin{aligned}
&\sum_{j=A,B}\Gamma^{j A}(f_j,g_A)\\
&\sum_{j=A,B}\Gamma^{j B}(f_{j},g_B)
\end{aligned}
\right],
\end{equation}
for $\mathbf{f}=[f_A,f_B]^T$ and $\mathbf{g}=[g_A,g_B]^T$. For brevity, we write $\mathbf{\Gamma}(\mathbf{f})=\mathbf{\Gamma}(\mathbf{f},\mathbf{f})$ in the sequel.

\subsection{Notations}
Throughout this paper, $C$ denotes a generic positive constant which may vary from line to line.  $C_a,C_b,\cdots$ denote the generic positive constants depending on $a,~b,\cdots$, respectively, which also may vary from line to line. $A\lesssim B$ means that there exists a constant $C>0$ so that $A\leq C B$ and $A\lesssim_{a}B$ means that the constant depends on $a$. For any vector or matrix $\mathbf{M}$, we denote $\mathbf{M}^T$ as its transpose. For any scalar function, we choose the convention to identify the space by the name of the concerned variables. For example, we denote $L^p_v:=L^p(\mathbb{R}^3,\dd v)$ and denote its norm as $|\cdot|_{L^p_v}$. Similarly, $L^p(\mathbb{R},\dd y)$ is denoted by $L^p_y$ while its norm is denoted by $|\cdot|_{L^{p}_y}$. 
For any vector-valued functions, the spatial-velocity mixed $L^p$ norm is denoted by $\|\cdot\|_{L^p}$. For simplicity, we use $\nu^{1/2}$ to denote the $2\times2$ diagonal matrix diag$(\nu_A^{1/2},\nu_B^{1/2})$. Similarly, the $2\times2$ diagonal matrix diag$(\nu_A^{-1/2},\nu_B^{-1/2})$ is denoted by $\nu^{-1/2}$. The same notations also apply to the bi-Maxwellians, such as
$\mathbf{M}^{-1/2}:=\text{diag}(M_{A}^{-1/2},M_B^{-1/2})$. The weighted norm $|\cdot|_{L^{2}_{\nu}}$ is defined by
$$
|\mathbf{f}|_{L^{2}_{\nu}}:=|\nu^{1/2}\mathbf{f}|_{L^{2}_{v}}.
$$

\subsection{Main results}
Let \begin{align}\label{w}
w_{q,\vartheta}(v)=e^{q|v|^2}(1+|v|^2)^{\vartheta}
\end{align}
be the velocity weight function and $\vep=c_0-s>0$ be the strength of the shock.

\begin{theorem}\label{thm1.1}
Let $-3<\gamma\leq1$, $\theta=\f{2}{3-\gamma}$, $\vartheta>2$, and $m_A\neq m_B$, then there exist positive constants $\varpi$, $q$ and $\vep_0>0$, such that for $0<\vep\leq \vep_0$, the stationary problem \eqref{1.2.1}, \eqref{1.2.2}, \eqref{1.2.4} and \eqref{1.2.7-1} admits a unique solution  $\mathbf{F}=[F_A,F_B]^T$, up to a spatial translation,  which satisfies
\begin{align}\label{1.0.2}
\bigg\|w_{q,\vartheta}\mathbf{M}^{-1/2}_-\pa_x^k[\mathbf{F}-\mathbf{M}_-]\bigg\|_{L^{\infty}}\lesssim_k\vep^{k+1},
\end{align}
where $k\geq 0$ is an arbitrary integer. 
It further holds that 
\begin{align}\label{1.0.3}
\bigg|\nu^{1/2}\mathbf{M}_-^{-1/2}\pa_x^k[\mathbf{F}-\mathbf{M}_{-}]\bigg|_{L^2_{v}}\lesssim_k \vep^{k+1}e^{-\vep\varpi|x|}+\vep^{k+2}e^{-\varpi|\vep x|^{\theta}}, x\leq 0.
\end{align}
and
\begin{align}\label{1.0.3-1}
\bigg|\nu^{1/2}\mathbf{M}_-^{-1/2}\pa_x^k[\mathbf{F}-\mathbf{M}_{+}]\bigg|_{L^2_{v}}\lesssim_k \vep^{k+1}e^{-\vep\varpi|x|}+\vep^{k+2}e^{-\varpi|\vep x|^{\theta}}, x\geq 0.
\end{align}
Moreover, for any integer $k\geq0$, $\pa^k_x\mathbf{F}$ is continuous away from $\mathbb{R}\times\{v_1=s\}$.
\end{theorem}

\begin{remark}
Our method of this paper also works for the case $m_A=m_B$ under which the two-species Boltzmann equations \eqref{1.1.1} reduce to the single one. The only difference is the disappearance of Carleman's representation \eqref{c4} when $m_A=m_B$; however,  this will not lead to any difficulty.
\end{remark}

\subsection{Relevant literature}
The construction of small-amplitude shock profiles for the cutoff Boltzmann equation under the hard sphere assumption was first done in \cite{NT} and \cite{CN} using the Lyapunov-Schmidt method. Later, in \cite{Liu-Yu}, an approach via the macro-micro decomposition around local Maxwellians was developed to establish the positivity of shock profiles by studying their large time asymptotic stability for the time-evolutionary problem. In a similar macro-micro framework with inducing a small Knudsen number, the hydrodynamic limit with shock waves of the Boltzmann equation was proved in \cite{Yu}, see also \cite{Yu1,Yu2}. Using the approximation by a local Maxwellian whose fluid quantities are determined by the Navier-Stokes shock, an alternative approach of proving the existence of shock profiles for the Boltzmann equation with hard sphere collisions in Sobolev spaces was given in \cite{Metivier-Zumbrun}. For the Landau equation with Coulomb potentials, to overcome the large-velocity  degeneracy in the hypocoercivity estimates, a regularization method on the collision operator was developed in \cite{Albritton-Bedrossian-Novack}, together with the general strategy in \cite{Metivier-Zumbrun}, to build the construction of the Landau shock profile. Interested readers may further refer to \cite{Wynter} for an extension to the non-cutoff Boltzmann equation. Through the center manifold analysis in the steady regime, a time-asymptotic method using the pointwise Green’s functions with the application to boundary layers and shock wave was developed in \cite{Liu-Yu1}. From the dynamical systems point of view, the center manifold was constructed under the reduction to the canonical form in \cite{Pogan-Zumbrun,Pogan-Zumbrun1,Zumbrun}. We also mention the relevant results \cite{BGS,Bernhoff-Golse} for half-space problems and phase transition, \cite{Caflisch,Caflisch-Liu} for large shocks of discrete velocity models, and \cite{Bouchut,Golse} for BGK-type models.

Our solution space for proving Theorem \ref{thm1.1} in this paper would be space-velocity weighted $L^2-L^\infty$ spaces. The $L^\infty$ control of the linearized Boltzmann operator is motivated by the $L^2-L^\infty$ interplay method in \cite{Guo} in the study of Boltzmann equation on bounded domains. A polynomial-exponential velocity weight and a decomposition on the linearized operator as introduced in \cite{SG} are used to control the collision terms. For the steady case with soft potential interactions, the loss of collision dissipation in large velocities causes essential trouble to bound the $L^\infty$ norm. In \cite{DHWZ}, the authors made a crucial observation to define a speeded backward bi-characteristic for soft potentials such that the new collision frequency in the steady case has a strictly positive lower bound. To be more generic, we construct in this paper the shock profile for a binary gas mixture where the masses of two-species particles can be disparate. In such case or  more general multi-species situation, we mention \cite{BBBG,Briant} for properties of the linearized collision operator, \cite{BD} for the spatially inhomogeneous well-posedness on the torus close to thermal equilibrium, \cite{BGPCS} for the one dimensional case, and \cite{DL-JSP} for the dynamic stability of rarefaction waves for the Vlasov-Poisson-Boltzmann system.

\subsection{Outline of the paper}
The rest of this paper is arranged as follows. In Section \ref{sec2}, we collect some preliminary estimates on collision operators. Section \ref{sec3} is devoted to constructing the shock profile with small amplitude and proving our main result Theorem \ref{thm1.1}. In the Appendix \ref{sec.app}, we list some basic estimates that have been used in the previous sections.   

\section{Preliminary estimates on collision operator}\label{sec2}
In this section, we show several estimates on the two-species collision operators for our later use. The following lemmas are concerned with the linearized collision operator $\textbf{L}_{\textbf{M}}$ defined in \eqref{1.2.10} along a fixed bi-Maxwellian $\textbf{M}=[M_A,M_B]^T$, where
$$
M_j=\f{m^{3/2}_j}{(2\pi)^{3/2}}
e^{-\f{m_j|v|^2}{2}},\qquad j=A,B.
$$

\begin{lemma}[\cite{ABT}] 
The operator $\mathbf{L}_{\mathbf{M}}$ is self-adjoint and non-negative definite. The kernel of $\mathbf{L}_{\mathbf{M}}$ is a 6-dimension space which is spanned by the following bases:
\begin{equation}\label{2.1.2}
\left\{\begin{aligned}
&\chi_{-1}=[M_A^{1/2},0]^T,\quad\chi_{0}=[0,M_B^{1/2}]^T,\\
&\chi_i=[m_Av_iM_A^{1/2},m_Bv_iM_B^{1/2}]^T,\quad i=1,2,3, \\
&\chi_4=[\f{m_A|v|^2-3}{2}M_A^{1/2},\f{m_B|v|^2-3}{2}M_B^{1/2}]^T.
\end{aligned}\right.
\end{equation}
\end{lemma}

The next lemma follows from direct calculations.

\begin{lemma}\label{leL}
The operator $\mathbf{L}_{\mathbf{M}}$ satisfies the decomposition
$\mathbf{L}_{\mathbf{M}}=\nu_{\mathbf{M}}(v)-K$. Here,
$$
\nu_{\mathbf{M}}=\left[
  \begin{array}{cc}
    \nu_A, & 0 \\
    0, & \nu_{B}\\
  \end{array}
\right]
$$ 
is the diagonal matrix with the elements given as
\begin{align}\label{2.1.2.1}
\nu_{i}(v)=\sum_{j=A,B}\int_{\mathbb{R}^3\times \mathbb{S}^2}
B^{ji}(|v-u|,\sigma)M_j(u)\dd u\dd\sigma,\quad i=A,B,
\end{align}
and the operator $K=[K_A,K_B]^T$ is denoted by
$$\begin{aligned}
{K}_{i}\mathbf{f}=&\sum_{j=A,B}\int_{\mathbb{R}^3\times \mathbb{S}^2}
B^{ji}(|v-u|,\sigma)M_j^{1/2}(u)\\
&\times[
M_j^{1/2}(u')f_{i}(v')+M_{i}^{1/2}(v')f_j(u')-M_{i}^{1/2}(v)f_j(u)]\dd u\dd\sigma,\quad i=A,B.
\end{aligned}$$
Moreover, we have
\begin{align}\notag
\nu_0(1+|v|)^{\gamma}\leq \nu_{i}(v) \leq \nu_1(1+|v|)^{\gamma}, \quad i=A,B,
\end{align}
for some positive constants $\nu_0,\nu_1>0$.
\end{lemma}

To the end, we denote $\nu=\nu_{\mathbf{M}_-}$ for simplicity.
We also introduce the projection operator $P$ as
\begin{align}\label{2.1.5}
P\mathbf{f}:=\sum_{j=-1}^{4}\langle\chi_j,\mathbf{f}\rangle\chi_j.
\end{align} 
We have the following important coercivity inequality. 

\begin{lemma}\label{spectralgap}
There exists a positive constant $c_0>0$ such that
\begin{align}\label{2.1.4}
\langle\mathbf{L}_{\mathbf{M}}\mathbf{f},\mathbf{f}\rangle\geq c_0\left|\nu^{1/2}(I-P)\mathbf{f}\right|_{{L^2_v}}^2.
\end{align}
\end{lemma}
\begin{remark}\label{remarkspectral}
The coercivity \eqref{2.1.4} is well-known in the context of mono-species Boltzmann equations, for both hard and soft molecule interactions, see \cite{G1} and \cite{M1} for instance. For the two-species Boltzmann equations, \cite{DJMZ} gives a constructive proof of \eqref{2.1.4} when $m_A=m_B$. Later, Briant and Daus \cite{BD} proved \eqref{2.1.4} when $m_A\neq m_B$. Notice that their results only cover the hard potential case $0\leq \gamma\leq 1$. For the sake of completeness, we will present the proof of \eqref{2.1.4} for soft potential $-3<\gamma<0$ at the end of this section.
\end{remark}

In a similar way to the mono-species case, we will define several decompositions of operator $K$. The first one is $K=K_+-K_-$, which is given in terms of its component as
\begin{equation}\label{gain}\begin{aligned}
{K}_{+,i}\mathbf{f}=&\sum_{j=A,B}\int_{\mathbb{R}^3}\int_{\mathbb{S}^2}B^{ji}(|v-u|,\sigma)M_{j}^{1/2}(u)\\
&\times[
M_j^{1/2}(u')f_{i}(v')+M_{i}^{1/2}(v')f_j(u')]\dd\sigma\dd u,\quad i=A,B,
\end{aligned}\end{equation}
and 
\begin{equation}\label{loss}
\begin{aligned}
{K}_{-,i}\mathbf{f}=&\sum_{j=A,B}\int_{\mathbb{R}^3}\int_{\mathbb{S}^2}B^{ji}(|v-u|,\sigma)M_{i}^{1/2}(v)M_{j}^{1/2}(u)f_j(u)\dd\sigma\dd u,\quad i=A,B.
\end{aligned}\end{equation}

To deal with the singularity for $-3<\gamma<0$, we introduce a smooth cut-off function $0\leq \chi_{m}(\tau)\leq 1,$ with $0< m\leq 1$, such that
$$
\chi_{m}(\tau)=0, \text{ for }\tau\leq m,\quad\chi_{m}(\tau)=1, \text{ for }\tau\geq 2m.
$$
Then we define $K_m=[K_{m,A},K_{m,B}]^T$ where
\begin{align}
K_{m,i}\mathbf{f}=&\sum_{j=A,B}\int_{\mathbb{R}^3}\int_{\mathbb{S}^2}B^{ji}(|v-u|,\sigma)(1-\chi_m(|v-u|))M_j^{1/2}(u)\nonumber\\
&\times[
M_{j}^{1/2}(u')f_{i}(v')+M_{i}^{1/2}(v')f_j(u')-M_{i}^{1/2}(v)f_j(u)]\dd\sigma\dd u,\quad i=A,B,\nonumber
\end{align}
and $K_r=K-K_m$. Firstly, we have the smallness and fast decay of $K_m$.

\begin{lemma}\label{leK}
Let $0\leq q<\min\{m_A,m_B\}$, then
\begin{align}\label{2.1.6}
|K_m\mathbf{f}|_{L^{\infty}_v}\lesssim_q m^{3+\gamma}e^{-\f{q|v|^2}{4}}|\mathbf{f}|_{L^{\infty}_v}.
\end{align}
\end{lemma}
\begin{proof}
	The inequality \eqref{2.1.6} follows from the fact that
	\begin{align*}
		|K_{m,i}\mathbf{f}|_{L^{\infty}_v}&\leq C|\mathbf{f}|_{L^{\infty}_v}\int_{\mathbb{R}^3}|v-u|^\gamma M_j^{1/2}(u)(1-\chi_m(|v-u|))\dd u\notag\\
		&\leq C|\mathbf{f}|_{L^{\infty}_v}e^{-\f{q|v|^2}{4}}\int_{|v-u|\leq m}|v-u|^\gamma e^{-\f{(m_j-q)|u|^2}{4}}\dd u\notag\\
		&\leq C_q m^{3+\gamma}e^{-\f{q|v|^2}{4}}|\mathbf{f}|_{L^{\infty}_v}.
	\end{align*}
\end{proof}

It is important to write $K_r$ as an integral operator. This takes advantage of the Carleman's representation, see Lemma \ref{lmA.2} in the Appendix.

\begin{lemma}\label{lmK} 
For $i\in\{A,B\}$, there exist functions $k^{ji}(v,\eta)$, $j\in\{A,B\}$, such that
\begin{align}\notag
K_{r,i}\mathbf{f}=\sum_{j=A,B}\int_{\mathbb{R}^3}k^{ji}_r(v,\eta)f_j(\eta)\dd \eta.
\end{align}
Moreover, there exists a positive constant $c_1$ depending only on $m_A$ and $m_B$, such that the following estimates hold:
\begin{align}\label{2.1.8}
|k_{r}^{ji}(v,\eta)|&\leq C_{\gamma}|v-\eta|^\gamma e^{-c_1|v|^2-c_1|u|^2}+\f{C_{\gamma} m^{\gamma-1}}{|v-\eta|(1+|v|+|\eta|)^{1-\gamma}}e^{-c_1|v-\eta|^2-\f{c_1||v|^2-|\eta|^2|^2}{|v-\eta|^2}},
\end{align}	
and
\begin{align}\label{2.1.9}
|k^{ji}_r(v,\eta)|&\leq C_{\gamma}|v-\eta|^\gamma e^{-c_1|v|^2-c_1|u|^2} +|v-\eta|^{\f{3-\gamma}2}e^{-c_1|v-\eta|^2-\f{c_1||v|^2-|\eta|^2|^2}{|v-\eta|^2}}.
\end{align}		
Furthermore, there exists a positive constant $q_1$, such that for $0\leq q \leq q_1$ and $\vartheta\geq 0$, it holds that
\begin{equation}\label{2.1.10}
\int_{\mathbb{R}^3}|k_r^{ji}(v,\eta)|\cdot \frac{(1+|v|^2)^{\vartheta}e^{q|v|^2}}{(1+|\eta|^2)^{\vartheta}e^{q|\eta|^2}}d\eta\leq C_\gamma m^{\gamma-1}(1+|v|)^{\gamma-1},
\end{equation}
and 	
\begin{align}\label{2.1.11}
\int_{\mathbb{R}^3}|k_r^{ji}(v,\eta)|\cdot \frac{(1+|v|^2)^{\vartheta}e^{q|v|^2}}{(1+|\eta|^2)^{\vartheta}e^{q|\eta|^2}} d\eta\leq C_\gamma (1+|v|)^{-1}.
\end{align}
\end{lemma}
\begin{proof}
Recall the Carleman representation in Lemma \ref{lmA.2}. If the terms in $K_{r,i}$ takes the forms of \eqref{c1}-\eqref{c3}, which have similar structures to the mono-species Boltzmann equation, one can use the arguments as in \cite{DHWY} to yield our estimate. Hence, we only consider \eqref{c4}. Define
\begin{align}\label{2.1.7-1}
k^{ji}(v,u'):=&\f{1}{|u'-v|}e^{-\left|\f{\sqrt{m_i}+\sqrt{m_j}}{\sqrt{m_i}-\sqrt{m_j}}\right|^2\tilde{V}^T\mathbf{U}\tilde{V}}
e^{-\f{|m_i-m_j|^2|v_{\perp}|^2}{4}}\nonumber
\\
&\quad\times\int_{|\eta|=\f{|v-u'|}{|m_i-m_j|}}\f{b^{ji}(\theta)\chi_m(|v-u(\eta,v,u')|)}{|v-u(\eta,v,u')|^{1-\gamma}}
e^{-\f{\sqrt{m_im_j}|\sqrt{m_im_j}\eta-z(v,u')|^2}{2}}\dd \eta.
\end{align}
Notice that $U$ defined in \eqref{DefU} is positive definite. Then
\begin{align}\label{2.1.7}
e^{-\left|\f{\sqrt{m_i}+\sqrt{m_j}}{\sqrt{m_i}-\sqrt{m_j}}\right|^2\tilde{V}^T\mathbf{U}\tilde{V}}\lesssim e^{-\lambda_{ij}\left[|u'-v|^2+\f{||u'|^2-|v|^2|^2}{|u'-v|^2}\right]},
\end{align}
for some positive constants $\lambda_{ij}$ depending only on $m_i$ and $m_j$. Now, by the collision geometry \eqref{1.1.3}, we have
$$|v-u'|=|\f{m_j(v-u)}{m_i+m_j}-\f{m_i|v-u|}{m_i+m_j}\sigma|\leq |v-u|.
$$
Then, by \eqref{2.1.7}, $k^{ji}$ is bounded by
$$C_{m_i,m_j}|u'-v|^{\gamma}e^{-\lambda_{ij}\left[|u'-v|^2+\f{||u'|^2-|v|^2|^2}{|u'-v|^2}\right]}e^{-\f{|m_i-m_j|^2|v_{\perp}|^2}{4}},
$$
where the universal constant $C_{m_i,m_j}$ depends only on $m_i$ and $m_j$. Now, we decompose $v=v_{\parallel}+v_{\perp}$, where
$$
v_{\parallel}=\f{( v, u'-v)}{|u'-v|}\f{u'-v}{|u'-v|}.
$$
Here, we have denoted $(\cdot,\cdot)$ as the Euclidean inner product in $\mathbb{R}^3.$ Firstly, it is direct to compute that
\begin{align}\label{2.1.7-2}
|v_{\parallel}|^2=\f{1}{4}\left|\f{|u'|^2-|v|^2}{|u'-v|}-|u'-v|\right|^2.
\end{align}
Then we have
$$
\begin{aligned}
\lambda_{ij}\bigg[|u'-v|^2&+\f{||u'|^2-|v|^2|^2}{|u'-v|^2}\bigg]+\f{|m_i-m_j|^2|v_{\perp}|^2}{4}\\
&\geq \f{\lambda_{ij}|u'-v|^2}{2}+\lambda_{ij}|v_{\parallel}|^2+\f{|m_i-m_j|^2|v_{\perp}|^2}{4}\\
&\geq \f{\tilde{\lambda}_{ij}}{2}|u'|^2+\f{\tilde{\lambda}_{ij}}{2}|v|^2.
\end{aligned}
$$
Here we have denoted $\tilde{\lambda}_{ij}:=\min\{\lambda_{ij},\f{|m_i-m_j|^2}{4}\}$. Therefore, $k^{ji}$ is further bounded by
$$C_{m_i,m_j}|u'-v|^{\gamma}e^{-\f{\tilde{\lambda}_{ij}}{2}|u'|^2-\f{\tilde{\lambda}_{ij}}{2}|v|^2}.
$$
This proves \eqref{2.1.8} and \eqref{2.1.9} for $k^{ji}$. Finally, \eqref{2.1.10} and \eqref{2.1.11} follows directly from these two bounds, see \cite{Glassey}.
\end{proof}

Now, we are ready to prove Lemma \ref{spectralgap}.

\begin{proof}[Proof of Lemma \ref{spectralgap}] 
By Remark \ref{remarkspectral}, we only need to consider the case $-3<\gamma<0$. We will follow the idea by Guo \cite{G1}. The key point is to decompose $K=K_c+K_s$, where $K_c$ is compact from $L^2_{\nu}\times L^{2}_{\nu}$ to $L^{2}_{\nu}$ and
$$
|\langle K_s \mathbf{f},\mathbf{f}\rangle|\leq {\eta}|\mathbf{f}|_{L^2_{\nu}}^2,
$$
for small $\eta>0$. This relies heavily on our Carlemann representation, see \eqref{c1}-\eqref{c4} in Appendix. Notice that the representations \eqref{c1}-\eqref{c3} are more or less the same as the formulas  in \cite[(23) and (36)]{G1}, while the form like \eqref{c4} does not appear. Thus, it suffices to consider \eqref{c4} in the two-species case. We define $k^{ji}_c(v,u'):=k^{ji}(v,u')\Fi_{\{|u'-v|+|v|\leq N\}}$, where $k^{ji}$ is given by \eqref{2.1.7-1}. We also denote
$$K^{ji}_cf_j:=\int_{\mathbb{R}^3}k^{ij}_c(v,u')f_j(u')\dd u' \text{ and } K^{ji}_sf_j:=\int_{\mathbb{R}^3}[k^{ij}-k^{ij}_c](v,u')f_j(u')\dd u'.$$
By \eqref{2.1.7}, $k^{ji}_{c}(v,u')$ is bounded by
$$
C_{m}\f{e^{-\lambda_{ij}\left[|u'-v|^2+\f{||u'|^2-|v|^2|^2}{|u'-v|^2}\right]}}{|u'-v|}\f{|u'-v|^2}{|m_i-m_j|^2}\lesssim
e^{-\lambda_{ij}\left[|u'-v|^2+\f{||u'|^2-|v|^2|^2}{|u'-v|^2}\right]}|u'-v|\Fi_{\{|u'-v|+|v|\leq N\}}.
$$
Then, it is direct to check that $K^{ji}_c$ is a Hilbert-Schmidt operator, and thus a compact operator from $L^2_{\nu}$ to $L^2_{\nu}$. Now, we show that $K^{ij}_s$ is small. Recall that $|u'-v|+|v|\geq N.$\\

Case 1: $|u'-v|\geq \f{1}{2}[|u'-v|+|v|]\geq \f{N}{2}.$ In this case, we also have $$|u'|\leq |u'-v|+|v|\leq 2|u'-v|.$$
Use this and exponential decay estimate \eqref{2.1.7} to bound this part of $K^{ij}_sf_j$ by
\begin{align}
C_{m}e^{-\f{\lambda_{ij}|v|^2}{4}-\f{\lambda_{ij}N^2}{16}}\int_{\mathbb{R}^3}|f_j(u')|e^{-\f{\lambda_{ij}|v-u'|^2}{4}-\f{\lambda_{ij}|u'|^2}{16}}|v-u'|
\dd u'\lesssim_m e^{-\f{\lambda_{ij}|v|^2}{4}-\f{\lambda_{ij}N^2}{16}}|\nu^{1/2}_jf_j|_{L^2_{v}}.\nonumber
\end{align}

Case 2: $|v|\geq \f{1}{2}[|u'-v|+|v|]\geq \f{N}{2}.$ As before, we decompose $v=v_{\parallel}+v_{\perp}$.
Notice that $|v_{\parallel}|^2+|v_{\perp}|^2=|v|^2\geq \f{1}{4}[|u'-v|+|v|]^2\geq\f{N^2}{4}$. We further split this part into
$$\begin{aligned}
&|v_{\parallel}|^2\geq \f{|v|^2}{2}\geq \f{1}{8}[|u'-v|+|v|]^2\geq \f{N^2}{8},\\
&|v_{\perp}|^2\geq \f{|v|^2}{2}\geq \f{1}{8}[|u'-v|+|v|]^2\geq \f{N^2}{8}.
\end{aligned}
$$
Then, over the first sub-domain, by using \eqref{2.1.7} and \eqref{2.1.7-2}, $K^{ij}_sf_j$ is bounded by
\begin{align}
C_{m}&\int_{\{|u'-v|+|v|\geq N\}}|f_j(u')|e^{-\lambda_{ij}\left[|u'-v|^2+\f{||u'|^2-|v|^2|^2}{|u'-v|^2}\right]}|v-u'|\dd u'\nonumber\\
&\lesssim_m
\int_{\{|u'-v|+|v|\geq N\}}|f_j(u')|e^{-2\lambda_{ij}|v_{\parallel}|^2}|v-u'|\dd u\nonumber\\
&\lesssim_m e^{-\f{\lambda_{ij}|v|^2}{4}-\f{\lambda_{ij}N^2}{16}}\int|f_j(u')|
e^{-\f{\lambda_{ij}|u'|^2}{16}-\f{\lambda_{ij}|u'-v|^2}{16}}|u'-v|\dd u'\lesssim e^{-\f{\lambda_{ij}|v|^2}{4}-\f{\lambda_{ij}N^2}{16}}|\nu^{1/2}_jf_j|_{L^2_{v}}.\nonumber
\end{align}
Over the second sub-domain, we use the exponential decay $e^{-\f{|m_i-m_j|^2|v_{\perp}|^2}{4}}$ (see \eqref{2.1.7-1}) to bound $K^{ij}_sf_j$ by
\begin{align}
C_m&e^{-\f{|m_i-m_j|^2|v|^2}{32}-\f{|m_i-m_j|^2N^2}{128}}\int |f_j(u')|e^{-\f{|m_i-m_j|^2|u'|^2}{128}-\f{|m_i-m_j|^2|u'-v|^2}{128}}|u'-v|\dd u'\nonumber\\
&\lesssim_m e^{-\f{|m_i-m_j|^2|v|^2}{32}-\f{|m_i-m_j|^2N^2}{128}}|\nu^{1/2}_jf_j|_{L^2_{v}}.\nonumber
\end{align}
In summary, we have proved that
$$\int_{\mathbb{R}^3}\big|k^{ji}_sf_j\cdot f_j\big|\dd v\leq \f{C_m}{N}|\nu^{1/2}_jf_j|_{L^2_{v}}^2.
$$
The proof of \eqref{2.1.4} now follows from the same arguments as in \cite[Lemma 3]{G1}. We omit it for brevity. 
\end{proof}

\section{Existence for shock profile}\label{sec3}
This section is devoted to construct the shock profile with small amplitude and prove our main result Theorem \ref{thm1.1}. Recall the strength of the shock $\vep=c_0-s\ll1$. Here, $c_0$ and $s$ are respectively the sound speed and shock speed defined as before. Motivated by \cite{CN}, we introduce a scaled variable $y=\vep x$ and denote $\mathbf{\tilde{f}}(y,v)=\vep^{-1}\mathbf{f}(x,v)$. Therefore, from \eqref{1.2.8} and \eqref{1.2.9}, the equation of $\mathbf{\tilde{f}}=\mathbf{\tilde{f}}(y)$ reads as
\begin{align*}
(v_1-s)\pa_y\mathbf{\tilde{f}}+\vep^{-1}\mathbf{L\tilde{f}}=\mathbf{\Gamma}(\mathbf{\tilde{f}}),
\end{align*}
with
\begin{align*}
\lim_{y\rightarrow-\infty}\mathbf{\tilde{f}}(y,v)=[0,0]^T,\quad \lim_{y\rightarrow +\infty}\mathbf{\tilde{f}}(y,v)=\vep^{-1}[\f{M_{A,+}-M_{A,-}}{\sqrt{M_{A,-}}},\f{M_{B,+}
-M_{B,-}}{\sqrt{M_{B,-}}}]^T:=\mathbf{\tilde{f}_{\infty}}.
\end{align*}
In what follows, we drop the tildes again for simplicity of notation and rewrite that
\begin{align}\label{3.0.1}
(v_1-s)\pa_y\mathbf{f}+\vep^{-1}\mathbf{Lf}=\mathbf{\Gamma}(\mathbf{f}),
\end{align}
with
\begin{align}\label{3.0.2}
\lim_{y\rightarrow-\infty}\mathbf{f}(y,v)=[0,0]^T,\quad \lim_{y\rightarrow +\infty}\mathbf{f}(y,v)=\mathbf{f_{\infty}}.
\end{align}
\subsection{Bifurcation equations and solution ansatz}
Following the idea in \cite{CN}, we start with the following approximated generalized eigenvalue problem associated to the profile equation \eqref{3.0.1}:
\begin{align}\label{3.1.1}
\mathbf{L}\phi_{\vep}+\vep\lambda(v_1-s)\phi_{\vep}=\vep^2\mu_{\vep},
\end{align}
with $\vep>0$, $\lambda\in \mathbb{R},$ and $\mu_{\vep}$ is bounded uniformly in $\vep$. The following result gives the existence of the solution $\phi_{\vep}$ to \eqref{3.1.1}. It can be proved along the same lines in \cite[Proposition 3.1]{CN}
with a minor modification. For completeness, we put the proof in the Appendix.

\begin{proposition}\label{prop2.1}
Let $0\leq q<\min\{m_A,m_B\}$ and $\vep>0$ be sufficiently small. Then the generalized eigenvalue problem \eqref{3.1.1} admits a solution pair $(\lambda,\phi_{\vep})$, which satisfies
\begin{align}\label{3.1.2}
\langle(v_1-s)\phi_{\vep},\phi_{\vep}\rangle=-\vep,\quad \langle(v_1-s)\chi_i,\phi_{\vep}\rangle=0,\quad i=-1,\cdots, 4,
\end{align}
\begin{align}\label{3.1.3}
\langle\chi_i,\mu_{\vep}\rangle=0,\quad i=-1,\cdots,4,
\end{align}
and
\begin{align}\label{3.1.4}
\left|\mu_{\vep}(v)e^{\f{q|v|^2}{4}}\right|\leq C_{q}.
\end{align}
Moreover, $\phi_{\vep}$ has the following expansion: $\phi_{\vep}=\phi_{0}+\vep\phi_1+\vep^2\Xi_{\vep}$. Here
\begin{align}\label{3.1.4-1}
\phi_0=\sqrt{\f3{10\b{\rho}_-}}[\chi_{-1}+\chi_{0}+c_0\chi_1+\f23\chi_4]:=\b{\alpha}\phi_0',
\end{align}
and
\begin{align}\label{3.1.4-2}
\phi_1=\b{\alpha}[\sum_{i=-1}^4\beta_j\chi_j+\phi_1'],
\end{align}
where $\beta_j$ for $j=-1,\cdots 4$ are constants, $\phi_1'\in (\text{Ker}\, \mathbf{L})^{\perp}$ solves 
\begin{align}\label{phi1prime}
\mathbf{L}\phi_1'=-\lambda(v_1-c_0)\phi_0',	
\end{align}
and $\Xi_{\vep}$ is uniformly bounded in the sense that
\begin{equation}\label{boundXi}
	\big|\Xi_{\vep}(v)e^{\f{q|v|^2}{4}}\big|\leq C_q.
\end{equation}
The constant $C_q>0$ is independent of $\vep$. Furthermore, the generalized eigenvalue $\lambda$ is independent of $\vep$ and can be explicitly computed as
\begin{align}\label{3.1.4-4}
	\lambda=\b{\alpha}^{-2}\langle(v_1-c_0)\phi_0',\mathbf{L}^{-1}(v_1-c_0)\phi_0'\rangle^{-1}>0.
\end{align}
\end{proposition}
For convenience, we denote 
\begin{equation}\label{Defvarphi}
\varphi_{\vep}:=\phi_1+\vep\Xi_{\vep},
\end{equation}
 in the sequel. One can rewrite the expansion of $\phi_\eps$ as
 \begin{equation}\label{expanphi}
 	\phi_{\vep}=\phi_{0}+\vep\phi_1+\vep^2\Xi_{\vep}=\phi_{0}+\vep\varphi_\varepsilon.
 \end{equation}
  Next, we define the following projection operator
\begin{align}\label{3.1.5}
\pi(\mathbf{f}):=\f{\langle(v_1-s)\phi_{\vep},\mathbf{f}\rangle}{\langle(v_1-s)\phi_{\vep},
\phi_{\vep}\rangle}\phi_{\vep}=-\vep^{-1}\langle(v_1-s)\phi_{\vep},\mathbf{f}\rangle\phi_{\vep},
\end{align}
associated to the approximated generalized eigenfunction $\phi_{\vep}$, and its adjoint operator
\begin{align}\label{3.1.6}
\pi_*(\mathbf{f}):=-(v_1-s)\phi_{\vep}\langle\varphi_{\vep},\mathbf{f}\rangle.
\end{align}
It is direct to check that $\pi^2=\pi$ and $\pi_*^2=\pi_*$. Moreover, as in \cite{CN}, we have the following properties of the operators above.
\begin{lemma}\label{lm2.2}
The projection operators $\pi$ and $\pi_*$ satisfy the following properties:\\

1. If $\langle(v_1-s)\chi_i, \mathbf{f}\rangle=0,$ for $i=-1,\cdots,4$, then it holds that
\begin{align}\label{3.1.7}
\pi_*(v_1-s)\mathbf{f}=(v_1-s)\pi \mathbf{f}.
\end{align}

2. If $\mathbf{g}\in \text{span}\{\chi_i\}_{i=-1}^4$, then it holds that
\begin{align}\notag
\langle\pi \mathbf{f},\mathbf{g}\rangle=\langle \mathbf{f},\pi_* \mathbf{g}\rangle.
\end{align}

3. For the projection $P$ defined in \eqref{2.1.5}, one has
\begin{align}\label{pi*L}
\pi_*\mathbf{Lf}=-\vep\lambda(v_1-s)\pi \mathbf{f}-\vep(v_1-s)\phi_{\vep}\langle\mu_{\vep},(I-P)\mathbf{f}\rangle.
\end{align}

4. It holds that $(I-\pi_*)\mathbf{Lf}=\mathbf{L}_1\mathbf{f}-\vep(I-\pi_*)\mu_{\vep}\langle(v_1-s)\phi_{\vep},\mathbf{f}\rangle$, where
\begin{align}\label{DefL1}
\mathbf{L}_1=(I-\pi_*)\mathbf{L}(I-\pi).
\end{align}

5. The operator $\mathbf{L}_1$ in \eqref{DefL1} is self-adjoint, non-negative, and
\begin{align}\label{3.1.8-1}
\text{Ker}\,\mathbf{L}_1=\text{span}\{\chi_{-2}, \chi_{-1},\cdots, \chi_4\},
 \end{align}
where we have denoted that $\chi_{-2}:=\varphi_{\vep}$ and $\chi_{-1},\cdots,\chi_4$ are defined in \eqref{2.1.2}.
\end{lemma}
We shall use the eigenprojection $\pi_*$ in \eqref{3.1.6} to decompose the solution into the fluid-like part and kinetic part. Before that, we need to clarify the asymptotic values of these two parts at $x=\pm\infty$. Recalling the downstream $\mathbf{f}_{\infty}$ defined in \eqref{3.0.2}, we have the following result.

\begin{lemma}\label{lm2.3}
If $0<\vep\ll1$, there exists a smooth function $\mathbf{g}_{\infty}$ with
\begin{align}\label{3.1.8-2}
\big|\mathbf{g}_{\infty}(v)e^{\f{q|v|^2}{4}}\big|\leq C_q,
\end{align}
for $0\leq q<\min\{m_A,m_B\}$, such that
\begin{equation}\label{finfty}
\mathbf{f}_{\infty}=-\f{3}{2c_0\b{\alpha}}\phi_{\vep}+\vep \mathbf{g}_{\infty}.
\end{equation}
 Moreover, we have the following decomposition of $\mathbf{f}_{\infty}$ associated to the projection $\pi$ that
\begin{align}\label{3.1.8-3}
\mathbf{f}_{\infty}=z_{\infty}\phi_{\vep}+\vep(I-\pi)\mathbf{g}_{\infty},\ \text{with}\  z_{\infty}=-\f{3}{2c_0\b{\alpha}}-\vep\langle(v_1-s)\varphi_{\vep},\mathbf{g}_{\infty}\rangle.
\end{align}

\end{lemma}
\begin{proof}
Expanding $M_{i,+}$ for $i=A,B$, we have
\begin{align}
M_{i,+}=M_{i,-}+\left\{\f{\rho_{i,+}-\rho_{i,-}}{\rho_{i,-}}
+m_{i}v_1u_++(\f{m_{i}|v|^2}{2}-\f32)(\theta_+-1)\right\}M_{i,-}+\vep^2 g_{i,\vep},\nonumber
\end{align}
for some smooth $g_{i,\vep}=g_{i,\vep}(v)$ with 
\begin{equation}\label{boundgeps}
	\big|g_{i,\vep}(v)M_{i,-}^{1-\eta}\big|\leq C_q,
\end{equation}
for any $0<\eta<1.$
From the far-field conditions \eqref{3.0.2}, we have a smooth function 
\begin{equation}\label{tildeginfty}
	\mathbf{\tilde{g}}_{\infty}=[g_{A,\vep}M_{A,-}^{-1/2},g_{B,\vep}M_{B,-}^{-1/2}]^T,
\end{equation}
such that
\begin{align}\label{3.1.8-4}
\mathbf{f}_{\infty}=-\f{3}{2c_0}(\chi_{-1}+\chi_0+c_0\chi_1+\f23\chi_4)+\vep \mathbf{{\tilde{g}}_{\infty}}=-\f3{2c_0\b{\alpha}}\phi_0+\vep \mathbf{\tilde{g}_{\infty}}.
\end{align}
Then we let $\mathbf{g}_{\infty}:=\mathbf{\tilde{g}_{\infty}}+\vep^{-1}\f3{2c_0\b{\alpha}}(\phi_0-\phi_{\vep})=\mathbf{\tilde{g}_{\infty}}-\f3{2c_0\b{\alpha}}\varphi_{\vep}$ to obtain \eqref{finfty}. Thus, the upper bound in \eqref{3.1.8-2} follows from \eqref{boundgeps}, \eqref{tildeginfty}, \eqref{Defvarphi}, \eqref{3.1.4-2} and \eqref{boundXi}. To further show \eqref{3.1.8-3}, we multiply $\chi_i$ on \eqref{3.0.1}, take integration on $v$ and use the orthogonal conditions \eqref{1.2.2.2} to obtain that
$$\f{\dd}{\dd y}\langle(v_1-s)\chi_i,\mathbf{f}\rangle=0,\quad i=-1,\cdots,4.
$$
Therefore, after integrating on $y$ from $-\infty$ to $+\infty$ and taking summation on $i$, it holds that
\begin{align}\label{orthofinfty}
\langle(v_1-s)\phi_0,\mathbf{f}_{\infty}\rangle=0.
\end{align}
Using \eqref{3.1.5}, \eqref{expanphi} and \eqref{finfty}, we have
\begin{align*}
\pi \mathbf{g}_{\infty}&=-\vep^{-1}\phi_{\vep}\langle(v_1-s)\phi_{0},\mathbf{g}_{\infty}\rangle-
\phi_{\vep}\langle(v_1-s)\varphi_{\vep},\mathbf{g}_{\infty}\rangle\notag\\
&=\vep^{-2}\phi_{\vep}\langle(v_1-s)\phi_{0},(\mathbf{f}_{\infty}+\f{3}{2c_0\b{\alpha}}\phi_{\vep})\rangle-
\phi_{\vep}\langle(v_1-s)\varphi_{\vep},\mathbf{g}_{\infty}\rangle,
\end{align*}
which, together with \eqref{orthofinfty}, \eqref{3.1.4-1} and \eqref{3.1.2}, yields
\begin{align}\label{piginfty}
	\pi \mathbf{g}_{\infty}&=-
	\phi_{\vep}\langle(v_1-s)\varphi_{\vep},\mathbf{g}_{\infty}\rangle.
\end{align}
Substitute the above equality into \eqref{3.1.8-4}, one gets \eqref{3.1.8-3}. Therefore, the proof of Lemma \ref{lm2.3} is complete.
\end{proof}
We will seek for the solution to \eqref{3.0.1} in the form of
\begin{align}\notag
\mathbf{f}(y,v)=z(y)\phi_{\vep}(v)+\vep \psi(y,v),
\end{align}
with orthogonal conditions \begin{align}\label{3.1.10}
\langle(v_1-s)\chi_i,\psi\rangle=0, 
\end{align}
for $i=-2,\cdots,4,$ 
and the following far-field conditions
\begin{equation}\label{Defphiinf}
\left\{\begin{aligned}
	&\lim_{y\rightarrow-\infty}[z(y),\psi(y,v)]=[0,0],\\ 
	&\lim_{y\rightarrow+\infty}[z(y),\psi(y,v)]=[z_{\infty},\psi_{\infty}(v)]=[z_{\infty},(I-\pi)\mathbf{g}_{\infty}(v)].
\end{aligned}\right.
\end{equation}
Now, we derive the equations of $z(y)$ and $\psi(y,v)$. Act the projection $\pi_*$ on \eqref{3.0.1} so as to obtain
\begin{align*}
\pi_*(v_1-s)\pa_y\mathbf{f}+\vep^{-1}\pi_*\mathbf{Lf}=\pi_*\mathbf{\Gamma}(\mathbf{f}),
\end{align*}
which, together with \eqref{3.1.7}, \eqref{pi*L}, yields
\begin{align*}
(v_1-s)\pa_y(\pi\mathbf{f})-\lambda(v_1-s)\pi \mathbf{f}-(v_1-s)\phi_{\vep}\langle\mu_{\vep},(I-P)\mathbf{f}\rangle=\pi_*\mathbf{\Gamma}(\mathbf{f}).
\end{align*}
Using the definition of $\pi$ in \eqref{3.1.5} and the properties \eqref{3.1.2} and \eqref{3.1.10}, it holds that
$$
\pi\mathbf{f}=\pi(z\phi_{\vep}+\vep\psi)=-\vep^{-1}z\langle(v_1-s)\phi_{\vep},\phi_{\vep}\rangle\phi_{\vep}-\langle(v_1-s)\phi_{\vep},\psi\rangle\phi_{\vep}=z\phi_{\vep}.
$$
We combine the above two identities to get
\begin{align}\label{3.1.11}
\pa_yz-\lambda z+\zeta_{\vep} z^2=\vep \tilde{r}_1,
\end{align}
where we have denoted
\begin{align}\label{def.veps}
\zeta_{\vep}=\langle\varphi_{\vep},\mathbf{\Gamma}(\phi_{\vep})\rangle,
\end{align}
and
\begin{align}\label{3.1.11.1}
\tilde{r}_1=\langle\mu_{\vep},(I-P)(z\varphi_{\vep}+\vep\psi)\rangle-\langle\varphi_{\vep},\mathbf{\Gamma}(z\phi_{\vep}+\vep\psi,\psi)
+\mG(\psi,z\phi_{\vep})\rangle.
\end{align}
To further investigate $\zeta_{\vep}$ in \eqref{def.veps}, we will need the following property of $\mG(\mathbf{g})$.
\begin{lemma}
For any $\mathbf{g}=[g_A,g_B]^T \in \text{Ker}\, \mathbf{L}$, we have \begin{align}\label{Gag}
	\mG(\mathbf{g})=\f12\mathbf{L}\left([\f{g_A^2}{\sqrt{M_{A,-}}},\f{g_B^2}{\sqrt{M_{B,-}}}]^T\right).
\end{align}	
\end{lemma}
\begin{proof} Assume that $g=\sum_{i=-1}^{4}\omega_i\chi_i$ with $\omega_i\in \mathbb{R}.$ For any $0<\vep'\ll1$, we introduce a parametrized bi-Maxwellian $\mathbf{M}_{\vep'}=[M_{A,\vep'},M_{B,\vep'}]^T$ with
$M_{i,\vep'}=\f{\rho_i(1+\vep'\omega_{0})}{[2\pi(1+\vep'\omega_4)]^{3/2}}
e^{-\f{m_{i}|v-\vep'\b{\omega}|}{2(1+\vep'\omega_4)}}$, $i=A,B$.
Here, we have denoted $\b{\omega}=(\omega_1,\omega_2,\omega_3)$. A direct computation shows that
$$\left(
\begin{array}{cc}
	M_{A,-}^{-1/2}& 0 \\
	0&M_{B,-}^{-1/2}  \\
\end{array}
\right)\f{\dd^2 \mathbf{Q}(\mathbf{M}_{\vep'})}{\dd \vep'^2}\bigg|_{\vep'=0}=2\mG(\mathbf{g})-\mathbf{L}\left([\f{g_A^2}{\sqrt{M_{A,-}}},\f{g_B^2}{\sqrt{M_{B,-}}}]^T\right).
$$
Since $\mathbf{Q}(\mathbf{M}_{\vep'})\equiv\mathbf{0}$,  the identity $\f{\dd^2\mathbf{Q}(\mathbf{M}_{\vep'})}{\dd \vep'^2}\bigg|_{\vep'=0}=\mathbf{0}$ directly yields \eqref{Gag}.
\end{proof}

Then we can prove the following result showing the exact value of the principle part for $\zeta_{\vep}$ in \eqref{def.veps}.

\begin{lemma}\label{lm2.4}
It holds that $\zeta_{\vep}$ can be represented as $\zeta_{\vep}=-\f23c_0\b{\alpha}\lambda+\vep\zeta'_{\vep}$ with $\zeta'_{\vep}$ being bounded uniformly in $\vep$. Here, $c_0$ is the sound speed, and $\b{\alpha}$ and $\lambda$ are defined in \eqref{3.1.4-1} and \eqref{3.1.4-4},  respectively.
\end{lemma}
\begin{proof} We see that
\begin{align}
\zeta_{\vep}=\b{\alpha}\langle\phi_1',\mG(\phi_0)\rangle
+\vep\big[\langle\Xi_{\vep},\mG(\phi_{\vep})\rangle+\langle\phi_1',\mG(\phi_{\vep},\varphi_{\vep})+
\mG(\varphi_{\vep},\phi_0)\rangle\big]
:=\zeta_0+\vep\zeta_{\vep}'.\nonumber
\end{align}
Using \eqref{Gag}, \eqref{2.1.2}, \eqref{3.1.4-1} and \eqref{phi1prime}, one directly computes that 
$
\zeta_0=-\f{2}{3}\b{\alpha}c_0\lambda.
$ This completes Lemma \ref{lm2.4}. 
\end{proof}

We further decompose $z(y)=z_0(y)+\vep z_1(y)$, where the principal part $z_0(y)$ satisfies the exact far-field conditions that
$$\lim_{y\rightarrow-\infty}z_0(y)=\lim_{y\rightarrow-\infty}z(y)=0, \quad \lim_{y\rightarrow+\infty}z_0(y)=\lim_{y\rightarrow+\infty}z(y)=z_{\infty},
$$
and $z_1(y)$ is the remainder with the homogeneous far fields. We denote
$$
\b{\lambda}=\zeta_{\vep}z_\infty\text{ and }
\lambda':=\vep^{-1}(\lambda-\b{\lambda})=\f3{2c_0\b{\alpha}}\zeta_\vep'+
\langle(v_1-s)\varphi_{\vep},\mathbf{g}_{\infty}\rangle\zeta_{\vep}.
$$
In the last equality above we have used \eqref{3.1.8-3} together with Lemma \ref{lm2.4}. Then, by \eqref{3.1.11}, the leading term $z_0$ satisfies the following Burgers equation:
\begin{equation}\label{3.1.12}
\left\{
\begin{aligned}
&\pa_yz_0-\b{\lambda}z_0+\zeta_{\vep}z_0^2=0,\\
&\lim_{y\rightarrow -\infty}z_0(y)=0,\quad \lim_{y\rightarrow +\infty}z_0(y)=z_{\infty},
\end{aligned}\right.
\end{equation}
and the equation for the remainder $z_1$ reads as
\begin{equation}\label{3.1.13}
\left\{
\begin{aligned}
&\pa_yz_1-\lambda z_1+2\zeta_{\vep}z_0z_1=-\vep\zeta_{\vep}z_1^2+\lambda' z_0+\tilde{r}_1:=r_1,\\
&\lim_{y\rightarrow\pm\infty}z_1(y)=0,
\end{aligned}
\right.
\end{equation}
where $\tilde{r}_1$ is given in \eqref{3.1.11.1}.
Now, from \eqref{3.1.12}, we can explicitly solve
\begin{equation}\label{z0}
z_0(y)=\f{1}{2}\{\tanh(\f12\b{\lambda} y)+1\}z_{\infty}.
\end{equation}
For convenience, we denote \begin{align}\label{p}
p(y):=\f{1}{2}\{\tanh(\f12\b{\lambda} y)+1\}.
\end{align}

Next, to derive the equation kinetic part $\psi(y,v)$, taking $\vep^{-1}(I-\pi_*)$ to \eqref{3.0.1}  and using Lemma \ref{lm2.2}, we obtain the following equation of $\psi$:
\begin{align}\label{3.1.14}
(v_1-s)\pa_y \psi+\vep^{-1}\mathbf{L}_1\psi=z(I-\pi_*)\mu_{\vep}+\vep^{-1}(I-\pi_*)\mathbf{\Gamma}(z\phi_{\vep}+\vep\psi).
\end{align}
We seek for $\psi$ in the form of $\psi=\psi_0+\vep^{1/2}\psi_1$, where $\psi_0$ is defined by
\begin{align}\label{3.1.14-1}
\psi_0(y,v):=p^2(y)\psi_{\infty}(v).
\end{align}
Note that the choice \eqref{3.1.14-1} is designed to eliminate the diverse inhomogeneous terms in \eqref{3.1.14}. From the fact that $\vep^{-1}\mathbf{Lf_{\infty}}=\mG(\mathbf{f}_{\infty},\mathbf{f}_{\infty})$, it holds that
$$
\vep^{-1}\mathbf{L}_1\psi_{\infty}=z_{\infty}(I-\pi_{*})\mu_{\vep}+\vep^{-1}(I-\pi_*)\mG(z_{\infty}\phi_{\vep}+\vep\psi_{\infty}).
$$
Combining the above equation with \eqref{3.1.14} and \eqref{3.1.14-1}, the remainder equation for $\psi_1$ reads as
\begin{equation}
	\left\{\begin{aligned}
		&(v_1-s)\pa_y \psi_1+\vep^{-1}\mathbf{L}_1\psi_1=\vep^{-1/2}[-(v_1-s)\pa_{y}\psi_0+(z-p^2z_{\infty})(I-\pi_*)\mu_{\vep}\notag\\
		&\qquad\qquad\qquad\qquad\qquad\qquad+\vep^{-1}(I-\pi_*)\big(\mathbf{\Gamma}(z\phi_{\vep}+\vep\psi)-p^2\mG(z_{\infty}\phi_{\vep}+\vep\psi_{\infty})\big)],\\
		&\lim_{y\rightarrow\pm\infty}\psi_1(y,v)=0,\qquad\langle (v_1-s)\chi_i,\psi_1\rangle=0,\qquad i=-2,\cdots,4,
	\end{aligned}\right.
\end{equation}
To further simplify the equation, notice from \eqref{z0} and \eqref{p} that 
$$
z-p^2z_{\infty}=p(1-p)z_{\infty}+\vep z_1.
$$
Moreover, it holds that
\begin{align*}
	&\mathbf{\Gamma}(z\phi_{\vep}+\vep\psi)-p^2\mG(z_{\infty}\phi_{\vep}+\vep\psi_{\infty})\notag\\
	&=\mathbf{\Gamma}(z\phi_{\vep}+\vep\psi,\vep z_1\phi_{\vep}+\vep^{3/2}\psi_{1})+\mathbf{\Gamma}(\vep z_1\phi_{\vep}+\vep^{3/2}\psi_{1},z_0\phi_{\vep}+\vep \psi_0)\notag\\
	&\quad+p^2\mathbf{\Gamma}( z_{\infty}\phi_{\vep}+\vep p\psi_{\infty},z_{\infty}\phi_{\vep}+\vep p\psi_{\infty})-p^2\mG(z_{\infty}\phi_{\vep}+\vep\psi_{\infty})\notag\\
	&=\mathbf{\Gamma}(z\phi_{\vep}+\vep\psi,\vep z_1\phi_{\vep}+\vep^{3/2}\psi_{1})+\mathbf{\Gamma}(\vep z_1\phi_{\vep}+\vep^{3/2}\psi_{1},z_0\phi_{\vep}+\vep \psi_0)\notag\\
	&\quad+p^2(p-1)\big[\mG\big(z_{\infty}\phi_{\vep}+\vep\psi_{\infty},\psi_{\infty}\big)+\mG\big(\psi_{\infty},z_{\infty}\phi_{\vep}+\vep p\psi_{\infty}\big)\big].
\end{align*}
We collect the above identities to rewrite the equation for $\psi_1$ as
\begin{equation}\label{3.1.15}
\left\{\begin{aligned}
&(v_1-s)\pa_y \psi_1+\vep^{-1}\mathbf{L}_1\psi_1=\vep^{-1/2}(\mathbf{J}_2+\mathbf{N}_2),\\
&\lim_{y\rightarrow\pm\infty}\psi_1(y,v)=0,\\
&\langle (v_1-s)\chi_i,\psi_1\rangle=0,\quad {i=-2,\cdots,4,}
\end{aligned}\right.
\end{equation}
where
\begin{align}\label{3.1.15-1}
\mathbf{J}_2=&-(v_1-s)\pa_{y}\psi_0+p(1-p)z_{\infty}(I-\pi_*)\mu_{\vep}\nonumber\\
&-p^2(1-p)(I-\pi_*)\big[\mG\big(z_{\infty}\phi_{\vep}+\vep\psi_{\infty},\psi_{\infty}\big)+\mG\big(\psi_{\infty},z_{\infty}\phi_{\vep}+\vep p\psi_{\infty}\big)\big],
\end{align}
and the nonlinear term
\begin{align}\label{3.1.15-2}
\mathbf{N}_2=&\vep^{1/2} z_1(I-\pi_*)\mu_{\vep}\notag\\
&+(I-\pi_*)\big[\mG(z\phi_{\vep}+\vep\psi,z_1\phi_{\vep}+\vep^{1/2}\psi_1)+\mG(z_1\phi_{\vep}+\vep^{1/2}\psi_1,
z_0\phi_{\vep}+\vep\psi_0)\big].
\end{align}
Here, we see $\mathbf{J}_2\sim O(1)$, thanks to our choice \eqref{3.1.14-1}. To solve \eqref{3.1.15}, one should take care of the kernel of $\mathbf{L}_1$ in \eqref{3.1.8-1}. Note that the conservation laws in $\eqref{3.1.15}_3$ kick the solution $\psi_1$ off Ker\,$\mathbf{L}_1$. Hence, following the idea in \cite{CN}, we introduce the following compensate operator $\mathbf{L}_2$ by
\begin{align}\label{DefL2}
	\mathbf{L}_2\psi:=\sum_{i=-2}^4\langle(v_1-s)\chi_i,\psi\rangle(v_1-s)\chi_i.
\end{align}
Let $\mathbf{H}:=\mathbf{L}_1+\mathbf{L}_2$, then we have the following result.

\begin{lemma}
The following three properties hold:

1. $\mathbf{H}$ is self-adjoint and positive definite. Moreover, there exists a positive constant $c_1$, such that
\begin{align}\label{3.1.16}
\langle \mathbf{H}\psi,\psi\rangle\geq c_1|\nu^{1/2}\psi|_{{L^2_v}}^2.
\end{align}

2. $\mathbf{H}=\nu-\b{K}$. Here, the diagonal operator $\nu$ is defined in \eqref{2.1.2.1} and $\b{K}$ can be decomposed into $\b{K}=\b{K}_m+\b{K}_r$, where the small part $\b{K}_m$ satisfies \eqref{2.1.6} after replacing $K_m$ by $\b{K}_m$, and $\b{K}_r=[\b{K}_{r,A},\b{K}_{r,B}]^{T}$ is an integral operator in the sense that
\begin{align}\notag
\b{K}_{r,i}\mathbf{f}=\sum_{j=A,B}\int_{\mathbb{R}^3}\b{k}^{ji}_r(v,\eta)f_j(\eta)\dd \eta,\quad i=A,B.
\end{align}
Moreover, the estimates \eqref{2.1.8}-\eqref{2.1.11} also hold after replacing the kernal $K^{ji}_{r}(v,\eta)$ by $\b{K}^{ji}_{r}(v,\eta)$.\\

3. If $\psi_1$ is the solution to \eqref{3.1.15}, then $\psi_1$ must solve
\begin{equation}\label{3.1.17}
\left\{\begin{aligned}
&(v_1-s)\pa_y \psi_1+\vep^{-1}\mathbf{H}\psi_1=\vep^{-1/2}(\mathbf{J}_2+\mathbf{N}_2),\\
&\lim_{y\rightarrow\pm\infty}\psi_1(y,v)=0.
\end{aligned}\right.
\end{equation}
\end{lemma}
\begin{proof}
    The third property holds by noticing that $\mathbf{L}_2\equiv0$ along the dynamics \eqref{3.1.15}.  For the second one, from the definition of $\mathbf{L}_1$ and $\mathbf{L}_2$ in \eqref{DefL1} and \eqref{DefL2}, we rewrite
    \begin{align*}
       \mathbf{H}\mathbf{f}=& (I-\pi_*)\mathbf{L}(I-\pi)\mathbf{f}+\mathbf{L}_2\mathbf{f}\notag\\
       =&\mathbf{L}\mathbf{f}+(v_1-s)\phi_{\vep}\langle\mathbf{L}\varphi_{\vep},\mathbf{f}\rangle+\vep^{-1}\langle(v_1-s)\phi_{\vep},\mathbf{f}\rangle\mathbf{L}\phi_{\vep}\notag\\
       &-\vep^{-1}(v_1-s)\langle(v_1-s)\phi_{\vep},\mathbf{f}\rangle\phi_{\vep}\langle\varphi_{\vep},\mathbf{L}\phi_{\vep}\rangle+\sum_{i=-2}^4\langle(v_1-s)\chi_i,\mathbf{f}\rangle(v_1-s)\chi_i\notag\\
       =&(\nu-K)\mathbf{f}+(v_1-s)\phi_{\vep}\langle\mathbf{L}\varphi_{\vep},\mathbf{f}\rangle+\langle(v_1-s)\phi_{\vep},\mathbf{f}\rangle\mathbf{L}\varphi_{\vep}\notag\\
       &-(v_1-s)\langle(v_1-s)\phi_{\vep},\mathbf{f}\rangle\phi_{\vep}\langle\varphi_{\vep},\mathbf{L}\varphi_{\vep}\rangle+\sum_{i=-2}^4\langle(v_1-s)\chi_i,\mathbf{f}\rangle(v_1-s)\chi_i.
    \end{align*}
    In the last equality above, we have used Lemma \ref{leL} and the fact that
    $$
    \mathbf{L}\phi_{\vep}=\mathbf{L}(\phi_0+\vep\varphi_{\vep})=\vep\mathbf{L}\varphi_{\vep}.
    $$
    Thus, the second property follows from Lemma \ref{leK} and the fast decay of $\chi_i$ and $\phi_\vep$ in \eqref{2.1.2}, \eqref{expanphi}, \eqref{3.1.4-1}, \eqref{3.1.4-2} and \eqref{boundXi}. Note that the new bounds for $\b{K}$ are independent of $\vep$. Then we can also prove \eqref{3.1.16} by the second property.
\end{proof}

In order to establish the equivalency between two problems \eqref{3.1.15} and \eqref{3.1.17}, we still need to prove that the solution to \eqref{3.1.17} solves \eqref{3.1.15}.

\begin{lemma}
	If $\psi_1$ is the solution to \eqref{3.1.17}, then $\psi_1$ must solve \eqref{3.1.15}.
\end{lemma}

\begin{proof}
	We only need to show that 
	\begin{align}\label{L2phi1}
	\mathbf{L}_2\psi_1=0.
	\end{align}
	To achieve this, we first multiply \eqref{3.1.17} by $\chi_i$ with $i=-2,\cdots,4$ and take integration on $v$ to get
	\begin{align}\label{psi1chii}
		\pa_y\langle(v_1-s) \psi_1,\chi_i\rangle +\vep^{-1}\langle \mathbf{L}_2\psi_1,\chi_i\rangle=\vep^{-1/2}\langle(\mathbf{J}_2+\mathbf{N}_2),\chi_i\rangle, 
	\end{align}
where we have used \eqref{3.1.8-1}. For the right hand side above,  it holds that
\begin{align*}
	\langle(v_1-s)\pa_{y}\psi_0,\chi_i\rangle=p^2\langle(v_1-s)\pa_{y}\psi_\infty,\chi_i\rangle=0,
\end{align*}
for $i=-1,\cdots,4$ by \eqref{3.1.10}, and 
\begin{align*}
	&\langle(v_1-s)\pa_{y}\psi_0,\chi_i\rangle=p^2\pa_{y}\langle(v_1-s)(I-\pi)\mathbf{g}_{\infty},\varphi_\vep\rangle\notag\\
	&=p^2\pa_{y}\langle(v_1-s)\mathbf{g}_{\infty},\varphi_\vep\rangle+p^2\pa_{y}\langle(v_1-s)
	\phi_{\vep}\langle(v_1-s)\varphi_{\vep},\mathbf{g}_{\infty}\rangle,\varphi_\vep\rangle\notag\\
	&=p^2\pa_{y}\langle(v_1-s)\mathbf{g}_{\infty},\varphi_\vep\rangle+\frac{1}{\vep}p^2\langle(v_1-s)
	\phi_{\vep},\phi_\vep\rangle\pa_{y}\langle(v_1-s)\varphi_{\vep},\mathbf{g}_{\infty}\rangle=0,
\end{align*}
by \eqref{Defphiinf}, \eqref{piginfty}, \eqref{expanphi}, \eqref{3.1.4-1} and \eqref{3.1.2}. Moreover, it follows from  \eqref{1.2.2.2} and \eqref{3.1.2} that
\begin{align*}
	\langle(I-\pi_*)G,\chi_i\rangle=\langle G,\chi_i\rangle+\langle (v_1-s)\phi_{\vep},\chi_i\rangle \langle\varphi_{\vep},G\rangle=0,
\end{align*}
for $i=-1,\cdots,4$, and we further have
\begin{align*}
	\langle(I-\pi_*)G,\phi_\vep\rangle=\langle G,\phi_\vep\rangle+\langle (v_1-s)\phi_{\vep},\phi_\vep\rangle \langle\varphi_{\vep},G\rangle=\langle G,\phi_\vep\rangle-\vep \langle\vep^{-1}(\phi_\vep-\phi_0),G\rangle=0,
\end{align*}
by \eqref{expanphi}, \eqref{3.1.4-1} and \eqref{3.1.2}. Thus, we see from the above identities, together with \eqref{3.1.15-1} and \eqref{3.1.15-2}, that the right hand side of \eqref{psi1chii} vanishes, which leads to 
	\begin{align*}
		&\pa_y\langle(v_1-s) \psi_1,\chi_i\rangle +\vep^{-1}\langle \mathbf{L}_2\psi_1,\chi_i\rangle\notag\\
	&=\pa_y\langle(v_1-s) \psi_1,\chi_i\rangle +\vep^{-1}\big({\sum_{j=-2}^4}\langle(v_1-s)\chi_j,\psi_1\rangle\big){\sum_{j=-2}^4}\langle (v_1-s)\chi_j,\chi_i\rangle\notag\\
	&=\pa_y\langle(v_1-s) \psi_1,\chi_i\rangle +\vep^{-1}{\sum_{j=-2}^4}\langle(v_1-s)\chi_j,\psi_1\rangle=0,
\end{align*}
for $i=-2,\cdots,4$. We further take summation on $i$ to deduce that
\begin{align*}
	\pa_y{\sum_{j=-2}^4}\langle(v_1-s) \psi_1,\chi_i\rangle +\vep^{-1}{\sum_{j=-2}^4}\langle(v_1-s)\psi_1,\chi_j\rangle=0,
\end{align*}
which, together with the condition $\lim_{y\rightarrow\pm\infty}\psi_1(y,v)=0$, yields
\begin{align*}
{\sum_{j=-2}^4}\langle(v_1-s)\psi_1,\chi_j\rangle=0.
\end{align*}
We finish our proof by combining the above identity with \eqref{DefL2} to get \eqref{L2phi1}.
\end{proof}
It remains to construct the solutions $[z_1(y),\psi_1(y,v)]^T$ to the coupling system \eqref{3.1.13} and \eqref{3.1.17}.
\subsection{Linear problem}
This part is devoted to solve the following linear problems:
\begin{equation}\label{L1}
\left\{
\begin{aligned}
&\pa_yz_1-\lambda z_1+2\zeta_{\vep}z_0z_1=r_1,\\
&\lim_{y\rightarrow\pm\infty}z_1(y)=0,
\end{aligned}
\right.
\end{equation}
and
\begin{equation}\label{L2}
\left\{\begin{aligned}
&(v_1-s)\pa_y \psi_1+\vep^{-1}\mathbf{H}\psi_1=\vep^{-1/2}\mathbf{r}_2,\\
&\lim_{y\rightarrow\pm\infty}\psi_1(y,v)=0.
\end{aligned}\right.
\end{equation}
We start with \eqref{L1}. Due to its shift invariance, we fix
\begin{align}\label{3.2.0}
z_1(0)=z_{1,0},
\end{align}
where $z_{1,0}$ is arbitrary in $ \mathbb{R}$. Then, the solvability of \eqref{L1} is given by the following Proposition.
\begin{proposition}\label{prop3.1}
Let $N\geq0$ be an integer, $\theta\in [0,1]$ and $\varpi> 0$. If $\theta=1$, we further assume that $0< \varpi\leq\f{\lambda}{4}$. For $\vep>0$ sufficiently small, if $\lim_{y\rightarrow\pm\infty}\f{\dd^kr_1}{\dd y^k}=0,\text{ for }k=0,\cdots N$ and $$\sum_{k=0}^N |e^{\varpi|\cdot|^{\theta}}\pa^{k}_yr_1|_{L^\infty_y}<\infty,$$ then \eqref{L1} with \eqref{3.2.0} admits a unique solution which satisfies
\begin{align}\label{3.2.0-1.2}
\sum_{k=0}^N|e^{\varpi|\cdot|^{\theta}}\pa^k_yz_1|_{L^{\infty}_y}\lesssim_{\varpi,\theta}|z_{1,0}|+\sum_{k=0}^N|e^{\varpi|\cdot|^{\theta}}\pa^{k}_yr_1|_{L^\infty_y}.
\end{align}
For $\theta>0$, if we further have $$\sum_{k=0}^N |e^{\varpi|\cdot|^{\theta}}\pa^{k}_yr_1|_{L^2_y}<\infty,
$$
then it holds that
\begin{align}\label{3.2.0-1.1}
\sum_{k=0}^N|e^{\varpi|\cdot|^{\theta}}\pa^k_yz_1|_{L^{2}_y}\lesssim_{\varpi,\theta}|z_{1,0}|+\sum_{k=0}^N|e^{\varpi|\cdot|^{\theta}}\pa^{k}_yr_1|_{L^2_y}.
\end{align}
\end{proposition}
\begin{proof}
We denote $A(y):=\lambda-2\zeta_{\vep}z_0(y)$. The equation \eqref{L1} with \eqref{3.2.0} can be solved by
\begin{align}\label{solutionz1}
z_1(y)=e^{\int_0^yA(y')\dd y'}z_{1,0}+\int_0^ye^{\int_\tau^yA(y')\dd y'}r_1(\tau)\dd \tau.
\end{align}
Notice that $\lim_{y\rightarrow+\infty} A(y)=\lambda-2\zeta_{\vep}z_{\infty}=-\lambda+O(\vep)<0$ for suitably small $\vep>0$. Then there exists $y_0>0$, such that for all $y>y_0$, $A(y)\leq -\f\lambda2$. Therefore, a direct computation shows, for any $0\leq \tau\leq y_0\leq y$, that
$$
e^{\int_{\tau}^yA(\tau_1)\dd \tau_1}=e^{\int_{\tau}^{y_0}A(\tau_1)\dd \tau_1+\int_{y_0}^{y}A(\tau_1)\dd \tau_1}\leq
e^{y_0|A|_{L^\infty_y}-\f{\lambda(y-y_0)}{2}}\leq e^{y_0(|A|_{L^\infty_y}+\f\lambda2)}e^{-\f{\lambda(y-\tau)}{2}}:=C_1e^{-\f{\lambda(y-\tau)}{2}}.
$$
Therefore, for $y>y_0$, it holds that
$$
\begin{aligned}
|z_1(y)|&\leq C_1e^{-\f{\lambda|y|}{2}}|z_{1,0}|+C_1\int_0^ye^{-\f{\lambda(y-\tau)}{2}}|r_1(\tau)|\dd \tau\\
&\leq C_{\varpi,\theta}e^{-\varpi |y|^{\theta}}|z_{1,0}|+C_1\sup_{\tau}|e^{\varpi|\tau|^{\theta}}r_1(\tau)|\int_0^ye^{-\f{\lambda(y-\tau)}{2}}e^{-\varpi |\tau|^{\theta}}\dd\tau\\
&\leq C_{\varpi,\theta}e^{-\varpi|y|^{\theta}}|z_{1,0}|+C_{\varpi,\theta}e^{-\varpi|y|^{\theta}}\sup_{\tau}|e^{\varpi|\tau|^{\theta}}r_1(\tau)|.
\end{aligned}
$$
In the last inequality above, we have used \eqref{5.9} and \eqref{5.10} from Lemma \ref{lmA.1} in the Appendix. Note that if $\theta=1$, our assumption $0< \varpi\leq\f{\lambda}{4}$ guarantees the inequality $e^{-\f{\lambda|y|}{2}}\leq e^{-\varpi|y|^{\theta}}$ and the integrability of $e^{-\f{\lambda(y-\tau)}{2}}e^{-\varpi |\tau|^{\theta}}$ on $\tau$. For $0\leq y\leq y_0$, it is direct to show that
\begin{align}\label{boundyLinfty}
|z_1(y)|&\leq (y_0+1)e^{y_0|A|_{L^\infty_y}}[|z_{1,0}|+\sup_{|\tau|\leq y_0}|r_1(\tau)|]\notag\\
&\leq(y_0+1)e^{y_0|A|_{L^\infty_y}+\varpi|y_0|^{\theta}}e^{-\varpi|y|^{\theta}}[|z_{1,0}|+\sup_\tau|e^{\varpi|\tau|^{\theta}}r_1(\tau)|]\notag\\
&\leq C_{\varpi,\theta}e^{-\varpi|y|^{\theta}}[|z_{1,0}|+\sup_\tau|e^{\varpi|\tau|^{\theta}}r_1(\tau)|].
\end{align}
We combine these two bounds to conclude that
\begin{align}\label{3.2.0.1}
|z_1(y)|\leq C_{\varpi,\theta}e^{-\varpi|y|^{\theta}}[|z_{1,0}|+|e^{\varpi|\cdot|^{\theta}}r_1|_{L^{\infty}_y}],
\end{align}
for any $y\geq 0$.
Since $\lim_{y\rightarrow-\infty}A(y)=\lambda>0$, a similar argument shows that \eqref{3.2.0.1} also holds for $y\leq 0$. This completes \eqref{3.2.0-1.2} for $N=0$. Similarly, the estimates on derivatives can be obtained inductively and hence omitted here for brevity.

 Next, we turn to the $L^2$ estimate \eqref{3.2.0-1.1}. When $0\leq y\leq y_0$, the boundedness leads to a similar structure of \eqref{boundyLinfty}. We directly have \eqref{3.2.0-1.1} for $k=0$ without giving the detailed proof. On the other hand, for $y>y_0$, by taking the inner product with $e^{2\varpi|y|^{\theta}}|z_1(y)|$ in \eqref{solutionz1}, we use H\"older's inequality, \eqref{5.9} and \eqref{5.10} to obtain 
 $$
 \begin{aligned}
 	|e^{\varpi|\cdot|^{\theta}}z_1|^2_{L^{2}_y}&\leq C\int^\infty_0e^{-\lambda|y|}e^{\varpi|y|^{\theta}}|z_{1,0}|^2\dd y+C\int^\infty_{y_0}\int_0^ye^{-\f{\lambda(y-\tau)}{2}}|r_1(\tau)|e^{2\varpi|y|^{\theta}}|z_1(y)|\dd \tau\dd y\notag\\
 	&\leq C_{\varpi,\theta}|z_{1,0}|^2+C\big(\int^\infty_{y_0}\int_0^ye^{-\f{\lambda(y-\tau)}{2}}|e^{\varpi|\tau|^{\theta}}r_1(\tau)|^2\dd \tau\dd y\big)^{1/2}\notag\\
 	&\qquad\qquad\qquad\qquad\times \big(\int^\infty_{y_0}\int_0^ye^{-\f{\lambda(y-\tau)}{2}}e^{-2\varpi|\tau|^{\theta}}e^{2\varpi|y|^{\theta}}|e^{\varpi|y|^{\theta}}z_1(y)|^2\dd \tau\dd y\big)^{1/2}\\
 	&\leq C_{\varpi,\theta}|z_{1,0}|^2+ C_{\varpi,\theta}\eta |e^{\varpi|\cdot|^{\theta}}z_1|^2_{L^{2}_y}+C_{\varpi,\theta,\eta}|e^{\varpi|\cdot|^{\theta}}r_1|_{L^2_y},
 \end{aligned}
 $$
for any $0<\eta<1$. Then, \eqref{3.2.0-1.1} holds by choosing $\eta$ sufficiently small. Hence, we get \eqref{3.2.0-1.1} for $N=0$ and the case $N>0$ can be proved inductively.
 
  Moreover, if $\theta=0,$ we have, for $y>0$, that
$$
\begin{aligned}
|z_{1}(y)|&\leq Ce^{-\f{\lambda|y|}{2}}|z_{1,0}|+C\int_0^ye^{-\f{\lambda(y-\tau)}{2}}|r_1(\tau)|\dd \tau\\
&\leq Ce^{-\f{\lambda|y|}{2}}|z_{1,0}|+C\int_0^ye^{-\f{\lambda\tau}{2}}|r_1(y-\tau)|\dd\tau\rightarrow 0\text{ as } y\rightarrow +\infty,
\end{aligned}
$$
by Lebesgue Convergence Theorem. Similarly, we can also show that $\lim_{y\rightarrow-\infty} z_1(y)=0$. The proof of Proposition \ref{prop3.1} is completed.
\end{proof}
Now we consider \eqref{L2}. Recall the velocity weight function $w_{q,\vartheta}(v)$ defined in \eqref{w}. The solvability of \eqref{L2} is given as follows.
\begin{proposition}\label{prop3.2}
Let $-3<\gamma\leq 1$ and $\vartheta>2$. There exist constants $\hat{q}_1>0$ and $\vep_1>0$, such that if $0<\vep\leq \vep_1$, $0\leq q\leq \hat{q}_1$, $$\lim_{y\rightarrow\pm\infty}\mathbf{r}_2(y,v)=\mathbf{0}, \text{ a.e for } v\in \mathbb{R}^3 \text{ and }
\|\nu^{-1/2}\mathbf{r}_2\|_{L^2}+\|\nu^{-1}w_{q,\vartheta}\mathbf{r}_2\|_{L^{\infty}}<\infty,$$ then \eqref{L2} admits a unique solution which satisfies
\begin{align}\label{3.2.0-2}
\vep^{-1/2}\|\nu^{1/2}\psi_1\|_{L^2}+\|w_{q,\vartheta}\psi_1\|_{L^{\infty}}\leq C\|\nu^{-1/2}\mathbf{r}_2\|_{L^2}+C\vep^{1/2}\|\nu^{-1}w_{q,\vartheta}\mathbf{r}_2\|_{L^{\infty}}.
\end{align}
Moreover, if $\mathbf{r}_2$ is continuous away from $\mathbb{R}\times\{v_1=s\}$, then $\psi_{1}$ is also continuous away from $\mathbb{R}\times\{v_1=s\}$.
\end{proposition}

Before going to the proof, we first show that the higher spatial regularity is given by the following corollary.

\begin{corollary}\label{cor3.2.5} Let $-3<\gamma\leq 1$, $\vartheta>2$, $0<\vep\leq \vep_1$, $0\leq q\leq \hat{q}_1$ and $N>0$ be an integer. If one has that for $k=0,\cdots N$, $\lim_{y\rightarrow \pm \infty}\pa^k_y\mathbf{r}_2(y,v)=\mathbf{0}\text{ a.e }v\in\mathbb{R}^3\text{ and }$
\begin{align}\label{3.2.0-6}
\sum_{k=0}^N\big\{\|\nu^{-1/2}\pa^k\mathbf{r}_2\|_{L^2}+\|\nu^{-1}w_{q,\vartheta}\pa^k\mathbf{r}_2\|_{L^{\infty}}\big\}<\infty,
\end{align}
then it holds that
\begin{align}\label{3.2.0-7}
\vep^{-1/2}\sum_{k=0}^N\|\nu^{1/2}\pa^k\psi_1\|_{L^2}+\|w_{q,\vartheta}\pa^k\psi_1\|_{L^{\infty}}
\leq C\sum_{k=0}^N\big\{\|\nu^{-1/2}\pa^k\mathbf{r}_2\|_{L^2}+\vep^{1/2}\|\nu^{-1}w_{q,\vartheta}\pa^k\mathbf{r}_2\|_{L^{\infty}}\big\}.
\end{align}
Moreover, for any $0\leq k\leq N$, if $\pa^k\mathbf{r}_2$ is continuous away from $\mathbb{R}\times \{v_1=s\},$ then $\pa^k\psi_1$ is also continuous away from $\mathbb{R}\times \{v_1=s\}$, for $0\leq k\leq N.$
\end{corollary}
\begin{proof} The estimates on derivatives can be shown by a quite standard induction argument. Assume that the estimate \eqref{3.2.0-7} holds when the regularity is up to $N$-order. It turns to show the bounds of $(N+1)$-order spatial derivatives. Denote $D^h\pa^N\psi_1$ as the difference quotient, namely,
$$
D^h\pa^N\psi_1(y,v):=\f{\pa^N\psi_1(y+h,v)-\pa^N\psi_1(y,v)}{h}.
$$
It is clear that $D^h\pa^N\psi_1$ satisfies \eqref{L2} with $\mathbf{r}_2$ replaced by the different quotient $D^h\pa^N\mathbf{r}_2$. Utilizing Proposition \ref{prop3.2}, $D^h\pa^N\psi$ satisfies the uniform-in-$h$ estimate \eqref{3.2.0-2} with $\mathbf{r}_2$ replaced by $D^h\pa^N\mathbf{r}_2$. Due to \eqref{3.2.0-6}, the estimate on $\pa^{N+1}\psi_1$ directly follows from taking $h\rightarrow 0^+$. The continuity of spatial derivatives of $\psi_1$ can be shown by a similar induction argument. Therefore, the proof of Corollary \ref{cor3.2.5} is completed.
\end{proof}

The remaining of this part is devoted to prove Proposition \ref{prop3.2}. We introduce the following auxiliary problem:
\begin{equation}\label{3.2.2}
\left\{
\begin{aligned}
&(v_1-s)\pa_y \psi+\vep^{-1}\nu(v)\psi=\vep^{-1}\delta\b{K}\psi+\vep^{-1/2}\mathbf{r}_2,\\
&\lim_{y\rightarrow\pm\infty}\psi(y,v)=0,
\end{aligned}
\right.
\end{equation}
with $0\leq \delta\leq 1$, $\psi=[\psi_A,\psi_B]^T$, and $\mathbf{r}_2=[r_{2,A},r_{2,B}]^T$. Denote $\mathcal{S}^{-1}_{\delta}$ to be the solution operator of \eqref{3.2.2}. We first construct $\mathcal{S}^{-1}_0$,
then derive the a priori estimate on $\mathcal{S}^{-1}_{\delta}$ which is uniform in $\delta$, and finally construct the solution operator $\mathcal{S}^{-1}_1$ of our target problem \eqref{L2} by the bootstrap argument. 

\begin{lemma}\label{lm3.2.1}
Let $\delta=0$ in \eqref{3.2.2}, $-3<\gamma\leq 1$, $\vep>0$, $q>0$ and $\vartheta>2$. If $$\lim_{y\rightarrow \pm \infty}\mathbf{r}_2(y,v)=\mathbf{0}\text{ a.e }v\in\mathbb{R}^3\text{ and }\|\nu^{-1/2}\mathbf{r}_2\|_{L^2}+\|\nu^{-1}w_{q,\vartheta}\mathbf{r}_2\|_{L^{\infty}}<\infty,$$ then there exists a unique solution $\psi$ to \eqref{3.2.2} which satisfies
\begin{align}\label{3.2.8}
\|\nu^{1/2}\psi\|_{L^2}\leq C \vep^{1/2}\|\nu^{-1/2}\mathbf{r}_2\|_{L^2}\text{ and }\|w_{q,\vartheta}\psi\|_{L^{\infty}}\leq C\vep^{1/2}\|\nu^{-1}w_{q,\vartheta}\mathbf{r}_2\|_{L^{\infty}}.
\end{align}
Here, $C>0$ is a universal constant independent of $\vep$. Moreover, if $\mathbf{r}_2$ is continuous away from $\mathbb{R}\times\{v_1=s\}$, then $\psi$ is also continuous away from $\mathbb{R}\times\{v_1=s\}$.
\end{lemma}

\begin{proof}
Firstly, the solution $\psi=[\psi_A,\psi_B]^T$ to \eqref{3.2.2} can be solved explicitly by
\begin{equation}\label{3.2.3}
\psi_{i}(y,v)=\left\{
\begin{aligned}
\vep^{-1/2}(v_1-s)^{-1}\int_{-\infty}^{y}e^{-\vep^{-1}\nu_{i}(v)(v_1-s)^{-1}(y-y')}r_{2,i}(y',v)\dd y',\quad &v_1-s>0,\\
-\vep^{-1/2}(v_1-s)^{-1}\int_{y}^{\infty}e^{-\vep^{-1}\nu_{i}(v)(v_1-s)^{-1}(y-y')}r_{2,i}(y',v)\dd y',\quad &v_1-s<0,
\end{aligned}
\right.
\end{equation}
where $i\in\{A,B\}$. To show \eqref{3.2.8}, we consider the weighted $L^2$-norm $\|\nu^{1/2}\psi\|_{L^2}$, for which the $v-$integration domain is split into two parts: $\{v_1-s>0\}\cup\{v_1-s<0\}$. For the first region, we have
\begin{align}\label{3.2.3-1}
\int_{-\infty}^ye^{-\vep^{-1}\nu_{i}(v)(v_1-s)^{-1}(y-y')}\dd y'\leq \vep(v_1-s)\nu_{i}^{-1}(v).
\end{align}
Thus, by H\"older's inequality, the $L^2$ norm of $\nu^{1/2}_i \psi_{i}$ over $\mathbb{R}\times\{v_1-s>0\}$ can be bounded by
$$
\begin{aligned}
&\int_{\mathbb{R}}\int_{\{v_1-s>0\}}(v_1-s)^{-1}\nu_i(v) \psi_{i}^2(y,v)\dd y\dd v\notag\\
&\leq \int_{\mathbb{R}}\dd y\int_{\{v_1-s>0\}}(v_1-s)^{-1}\dd v
\int_{-\infty}^ye^{-\vep^{-1}\nu_{i}(v)(v_1-s)^{-1}(y-y')}r_{2,i}^2(y',v)\dd y'\\
&\leq \int_{\mathbb{R}}\dd y'\int_{\{v_1-s>0\}}(v_1-s)^{-1}r_{2,i}^2(y',v)\dd v\int_{y'}^{+\infty}
e^{-\vep^{-1}\nu_{i}(v)(v_1-s)^{-1}(y-y')}\dd y\\
&\leq C\vep\|\nu^{-1/2}\mathbf{r}_2\|_{L^2}^2.
\end{aligned}
$$
Similarly, the integral of $\nu_i\psi_{i}^2$ over $\mathbb{R}\times\{v_1-s<0\}$ is bounded by
$C\vep \|\nu^{-1/2}\mathbf{r}_{2}\|_{L^2}^2.
$
This completes the weighted $L^2$-estimates in \eqref{3.2.8}. The $L^{\infty}$-estimate can be directly obtained from \eqref{3.2.3-1} and \eqref{3.2.3}. Finally, from the explicit formula \eqref{3.2.3}, it is direct to show that $\psi$ is continuous away from $\mathbb{R}\times\{v_1=s\}$, provided that $\mathbf{r}_2$ is continuous away from $\mathbb{R}\times\{v_1=s\}$. Therefore, the proof of Lemma \ref{lm3.2.1} is completed.
\end{proof}

Next, we derive the uniform-in-$\delta$ estimates on the solution to \eqref{3.2.2}. Firstly, we have the following uniform $L^2$-estimate.

\begin{lemma}\label{lm3.2.2}
Let $-3<\gamma\leq 1$, $0\leq \delta\leq 1$, $\vep>0,$ $q\geq 0$ and $\vartheta>2$. If $\psi=\psi(y,v)$ is a solution to \eqref{3.2.2} and satisfies
\begin{align}\label{3.2.10.2}
\|\nu^{1/2}\psi\|_{L^2}+\|w_{q,\vartheta}\psi\|_{L^{\infty}}<\infty,
\end{align}
then it holds that
\begin{align}\label{3.2.11}
 \|\nu^{1/2}\psi\|_{L^2}\leq C\vep^{1/2} \|\nu^{-1/2}\mathbf{r}_2\|_{L^2},
\end{align}
with a positive constant $C$ independent of $\vep$ and $\delta.$
\end{lemma}

\begin{proof} 
Taking the inner product of \eqref{3.2.2} with $\psi$ over $\mathbb{R}\times \mathbb{R}^3$, we obtain
\begin{align}\label{3.2.12}
\int_{\mathbb{R}}\langle(v_1-s)\pa_y\psi,\psi\rangle\dd y+\vep^{-1}\int_{\mathbb{R}}\big[\delta\langle \mathbf{H}\psi,\psi\rangle+(1-\delta)\langle\nu \psi,\psi\rangle\big]\dd y=\vep^{-1/2}\int_{\mathbb{R}}\langle \mathbf{r}_2,\psi\rangle\dd y.
\end{align}
Here, we have used the fact that $\nu-\delta\b{K}=\delta \mathbf{H}+(1-\delta)\nu$. By Cauchy-Schwarz inequality, the right hand side of \eqref{3.2.12} can be bounded by
\begin{align}\label{r2psi}
\vep^{-1/2}\int_{\mathbb{R}}|\langle \mathbf{r}_2,\psi\rangle|\dd y\leq \eta\vep^{-1} \|\nu^{1/2}\psi\|_{L^2}^2+C_{\eta}\|\nu^{-1/2}\mathbf{r}_2\|_{L^2}^2,
\end{align}
where $\eta>0$ can be arbitrarily small. By the positivity of $\mathbf{H}$ in \eqref{3.1.16}, for some $c_1,\nu_0>0,$ the second term on the left hand side of \eqref{3.2.12} has a lower bound as
\begin{align}\label{Kpsi}
\vep^{-1}\int_{\mathbb{R}}\big[\delta\langle \mathbf{H}\psi,\psi\rangle+(1-\delta)\langle\nu \psi,\psi\rangle\big]\dd y\geq\vep^{-1}[\delta c_1+(1-\delta)\nu_0]\|\nu^{1/2}\psi\|_{L^2}^2.
\end{align}
The first term on the left hand side of \eqref{3.2.12} can be computed as
\begin{align}\label{psipsi}
	\int_{\mathbb{R}}\langle(v_1-s)\pa_y\psi,\psi\rangle\dd y=\sum_{i=A,B}\int_{\mathbb{R}}\int_{\mathbb{R}^3}\pa_y\f{(v_1-s)\psi_i^2}{2}\dd v\dd y.
	\end{align}
Recalling the bounds  in \eqref{3.2.10.2} and  the fact that $\lim_{y\rightarrow\pm\infty}\psi_{i}(y,v)=0$, we have
$$
|(v_1-s)\psi^2_i(y,v)|\leq C|\psi_i(y,v)|\cdot\|w_{q,\vartheta}\psi\|_{L^{\infty}}\rightarrow 0,\quad \text{as }y\rightarrow\pm\infty, $$
and
$$\bigg|\int_{\mathbb{R}^3}\f{(v_1-s)\psi_i^2}{2}\dd v\bigg|\leq C\|w_{q,\vartheta}\psi\|_{L^\infty}^2\cdot\int_{\mathbb{R}^3}(1+|v|)^{-4}\dd v<\infty.
$$
Then, by Lebesgue Convergence Theorem, we have
\begin{align}\label{LCT}\sum_{i=A,B}\int_{\mathbb{R}}\int_{\mathbb{R}^3}\pa_y\f{(v_1-s)\psi_i^2}{2}\dd v\dd y=\sum_{i=A,B}\lim_{y\rightarrow \pm\infty}\int_{\mathbb{R}^3}\f{(v_1-s)\psi_i^2(y,v)}{2}\dd v=0.
\end{align}
Hence, $\eqref{3.2.11}$ holds by collecting \eqref{3.2.12}, \eqref{r2psi}, \eqref{Kpsi}, \eqref{psipsi} and \eqref{LCT}, and taking $\eta>0$ suitably small. The proof of Lemma \ref{lm3.2.2} is completed.
\end{proof}

Now, we are in the position to derive the crucial $L^\infty$-estimate. To do this, we define two kinds of 1-d characteristics serving to models with
hard $(0\leq \gamma\leq 1)$ and soft $(-3< \gamma <0)$ interactions respectively.
\begin{definition}(Characteristics for $0\leq \gamma\leq 1$)
For any $(t,y,v)\in [0,\infty)\times \mathbb{R}\times \mathbb{R}^3$, define $[Y^H(\tau;t,y,v),V^{H}(\tau;t,y,v)]$ as the solution to the following ODE
\begin{align}\notag
\begin{cases}
\frac{\dd Y^H(\tau)}{\dd\tau}=V_1^H(\tau)-s,\\[2mm]
\frac{\dd V^H(\tau)}{\dd\tau}=0,\\[2mm]
[Y^H(t),V^H(t)]=[y,v],
\end{cases}
\end{align}
that is,
\begin{align}\notag
[{Y}^H(\tau;t,y,v),V^H(\tau;t,y,v)]=[y-\tilde{v}_1(t-\tau),v] \quad\mbox{where}\quad \tilde{v}_1:=(v_1-s).
\end{align}
\end{definition}
\begin{definition} (Characteristics for $-3<\gamma<0$) For any $(t,y,v)\in [0,\infty)\times \mathbb{R}\times \mathbb{R}^3$, define $[Y^S(\tau;t,y,v),V^{S}(\tau;t,y,v)]$ as the solution to the following ODE
\begin{align}\notag
\begin{cases}
\frac{\dd Y^S(\tau)}{\dd\tau}=(\varrho+|V^S(\tau)|^2)^{\frac{|\gamma|}{2}}(V_1^S(\tau)-s):=\hat{V}_1^S(\tau),\\[2mm]
\frac{\dd V^S(\tau)}{\dd\tau}=0,\\[2mm]
[Y^S(t),V^S(t)]=[y,v],
\end{cases}
\end{align}
that is,
\begin{align}\notag
[{Y}^S(\tau;t,y,v),V^S(\tau;t,y,v)]=[y-\hat{v}_1(t-\tau),v],\quad\mbox{with}\quad \hat{v}_1:=(\varrho+|v|^2)^{\f{|\gamma|}{2}}(v_1-s).
\end{align}
Here $\varrho>1$ is a constant to be chosen later.
\end{definition}

Along the 1-d characteristics above, we can define the mild formulation of \eqref{3.2.2} which enables us to get $L^{\infty}$-estimate. Let $h_{i}:=w_{q,\vartheta}\psi_{i}$, $i\in \{A,B\}$, where  $w_{q,\vartheta}$ is defined in \eqref{w}. Then, the equation for $\mathbf{h}=[h_{A}, h_{B}]^T$ reads as
\begin{align}\label{3.2.13}
(v_1-s)\pa_y\mathbf{h}+\vep^{-1}\nu\mathbf{h}=\vep^{-1}\delta w_{q,\vartheta}\b{K}\left(\f{\mathbf{h}}{w_{q,\vartheta}}\right)+\vep^{-1/2}w_{q,\vartheta}\mathbf{r}_2, \quad y\in \mathbb{R}, v\in \mathbb{R}^3.
\end{align}

\begin{lemma}
Let $(t,x,v)\in [0,\infty)\times \mathbb{R}\times \mathbb{R}^3$ and $i\in\{A,B\}$. For $0\leq \gamma\leq 1$, we have
\begin{align}\label{3.2.14}
h_i(y,v)=&e^{-\vep^{-1}\nu_{i}(v)t}h_{i}(y-\tilde{v}_1t,v)+\vep^{-1/2}\int_0^te^{-\vep^{-1}\nu_{i}(v)(t-\tau)} w_{q,\vartheta}r_{2,i}
(Y^H(\tau),V^H(\tau))\dd
\tau\nonumber\\
&+\vep^{-1}\delta\int_0^te^{-\vep^{-1}\nu_{i}(v)(t-\tau)} w_{q,\vartheta}\b{K}_m\left(\f{h_{i}}{w_{q,\vartheta}}\right)
(Y^H(\tau),V^H(\tau))\dd
\tau\nonumber\\
&+\vep^{-1}\delta\sum_{j=A,B}\int_0^t\int_{\mathbb{R}^3}e^{-\vep^{-1}\nu_{i}(v)(t-\tau)} w_{q,\vartheta}\tilde{K}_r^{ji}(v,u)h_j
(Y^H(\tau),u)\dd u\dd
\tau.
\end{align}
And for $-3<\gamma<0$, denoting $\hat{\nu}_{i}(v):=(\varrho+|v|^2)^{\f{|\gamma|}{2}}\nu_{i}(v)$, it holds that
\begin{align}\label{3.2.15}
h_i(y,v)=&e^{-\vep^{-1}\hat{\nu}_{i}(v)t}h_{i}(y-\hat{v}_1t,v)\nonumber\\
&+\vep^{-1}\delta\int_0^te^{-\vep^{-1}\hat{\nu}_{i}(v)(t-\tau)} w_{q,\vartheta}(\varrho+|v|^{2})^{\f{|\gamma|}{2}} \b{K}_m\left(\f{h_{i}}{w_{q,\vartheta}}\right)
(Y^S(\tau),V^S(\tau))\dd
\tau\nonumber\\
&+\vep^{-1}\delta\sum_{j=A,B}\int_0^t\int_{\mathbb{R}^3}e^{-\vep^{-1}\hat{\nu}_{i}(v)(t-\tau)} (\varrho+|v|^{2})^{\f{|\gamma|}{2}}\tilde{k}_r^{ji}(v,u)h_j
(Y^S(\tau),u)\dd u\dd
\tau.\nonumber\\
&+\vep^{-1/2}\int_0^te^{-\vep^{-1}\hat{\nu}_{i}(v)(t-\tau)}(\varrho+|v|^{2})^{\f{|\gamma|}{2}} w_{q,\vartheta}r_{2,i}(Y^S(\tau),V^S(\tau))\dd
\tau:=\sum_{l=1}^4I_l.
\end{align}
In \eqref{3.2.14} and \eqref{3.2.15}, we have denoted
$$
\tilde{k}_r^{ji}(v,u)=\b{k}_r^{ji}(v,u)\f{w_{q,\vartheta}(v)}{w_{q,\vartheta}(u)}.
$$
\end{lemma}
\begin{lemma}\label{lm3.2.5}
Let $-3<\gamma \leq 1,$ $0\leq q\leq q_1$, where $q_1$ is defined in Lemma \ref{lmK}, 
and $\theta>2$. 
Assume that $\mathbf{h}=\mathbf{h}(y,v)\in L^{\infty}\times L^{\infty} $ is a solution to \eqref{3.2.13}. Then there exists $\vep_2>0$, such that for any $0<\vep\leq \vep_2$,
\begin{align}\label{3.2.16}
\|\mathbf{h}\|_{L^{\infty}}\leq C\vep^{1/2}\|\nu^{-1}w_{q,\vartheta}\mathbf{r}_2\|_{L^{\infty}}+C\|\nu^{-1/2}\mathbf{r}_2\|_{L^2}.
\end{align}
Here, the constant $C>0$ is independent of $\vep$.
\end{lemma}

\begin{proof}
Consider $-3<\gamma<0$. Notice that $\hat{\nu}_i(v)=(\varrho+|v|^2)^{\f{|\gamma|}{2}}\nu_{i}(v)\geq \hat{\nu}_0>0$ for some positive constant $\hat{\nu}_0$. Then, for $I_i$ defined in \eqref{3.2.15}, it holds that
\begin{align}\label{3.2.16-1}
|I_1|\leq e^{-\vep^{-1}\hat{\nu}_0t}\|\mathbf{h}\|_{L^{\infty}},
\end{align}
and
\begin{align}\label{3.2.16-2}
|I_4|\leq C\vep^{1/2}\|\nu^{-1}w_{q,\vartheta}\mathbf{r}_2\|_{L^{\infty}}.
\end{align}
For $I_2$, it holds from \eqref{2.1.6} that
\begin{align}\label{3.2.16-3}
|I_2|\leq Cm^{3+\gamma}e^{-\f{m_0|v|^2}{4}}\|\mathbf{h}\|_{L^{\infty}},
\end{align}
where we have denoted $m_0:=\frac{1}{2}\min\{m_A,m_B\}$. Substitute \eqref{3.2.16-1}, \eqref{3.2.16-2} and \eqref{3.2.16-3} into \eqref{3.2.15} to get
\begin{align}\label{3.2.16-4}
|h_i(y,v)|&\leq |I_3|+C\bigg\{e^{-\vep^{-1}\hat{\nu}_0t}+m^{3+\gamma}e^{-\f{m_0|v|^2}{4}}\bigg\}\|\mathbf{h}\|_{L^{\infty}}
+C\vep^{1/2}\|\nu^{-1}w_{q,\vartheta}\mathbf{r}_2\|_{L^{\i}}\nonumber\\
&:=|I_3|+A(t,v).
\end{align}
It remains to estimate $I_3$. Denote $y_1:=y-\hat{v}_1(t-\tau)$ and
$$
[Y_{1}^S(\tau_1),V^S_1(\tau_1)]:=[Y(\tau_1;\tau,Y(\tau),u),V({\tau_1;\tau,Y(\tau),u})]=[y_1-(\tau-\tau_1)\hat{u}_1,u].
$$
Then, we substitute the mild form of $h_j(Y^{S}(\tau),u)$ from \eqref{3.2.15} into $I_3$ and use \eqref{3.2.16-4} to obtain
\begin{align}\label{3.2.16-5}
I_3&\lesssim\vep^{-1}\sum_{j=A,B}
\int_0^{t}\int_{\mathbb{R}^3}e^{-\vep^{-1}\hat{\nu}_0(t-\tau)}(\varrho+|v|^2)^{\f{|\gamma|}{2}}|\tilde{k}^{ji}_r(v,u)h_j(y-\hat{v}_1(t-\tau),u)|\dd u\dd\tau\nonumber\\
&\lesssim_{\varrho} C\vep^{-1}(1+|v|^2)^{\f{|\gamma|}{2}}\sum_{j=A,B}\int_0^{t}\int_{\mathbb{R}^3}e^{-\vep^{-1}\hat{\nu}_0(t-\tau)}
|\tilde{k}^{ji}_r(v,u)A(\tau,u)|\dd u\dd\tau\nonumber\\
&\quad+C\vep^{-2}\sum_{j=A,B}\sum_{l=A,B}\int_0^te^{-\vep^{-1}\hat{\nu}_0(t-\tau)}\dd\tau \int_0^{\tau}\int_{\mathbb{R}^3}\int_{\mathbb{R}^3}U(\tau_1,u',u;\tau,v)\dd u'\dd u\dd\tau_1\nonumber\\
&:=B_1+B_2,
\end{align}
where
$$
U(\tau_1,u',u;\tau,v):=e^{-\vep^{-1}\hat{\nu}_0(\tau-\tau_1)}
(1+|v|^2)^{\f{|\gamma|}{2}}(1+|u|^2)^{\f{|\gamma|}{2}}\big|\tilde{k}^{ji}(v,u)
\tilde{k}^{lj}(u,u')h_{l}(Y_1^S(\tau_1),V_1^S(\tau_1))\big|.
$$
From \eqref{2.1.10} and \eqref{2.1.11}, one gets that
\begin{align}\label{3.2.16-6}
|B_1|\leq C\big\{m^{\gamma-1}e^{-\vep^{-1}\hat{\nu}_0t}
+m^{3+\gamma}\big\}\|\mathbf{h}\|_{L^{\infty}}+Cm^{\gamma-1}\vep^{1/2}\|\nu^{-1}w_{q,\vartheta}\mathbf{r}_2\|_{L^{\infty}}.
\end{align}
As for $B_2$, when $|v|>N$, it holds from \eqref{2.1.10} that
\begin{align}\label{3.2.16-7}
|B_2|\leq \f{Cm^{2(\gamma-1)}}{N}\|\mathbf{h}\|_{L^{\infty}}.
\end{align}
When $|v|\leq N$, we split the integral domain of $U$ with respect to $\dd \tau_1\dd u\dd u'$ into the following four parts: $$
\begin{aligned}\{|u|>2N\}&\cup\{|u|\leq 2N, |u'|>3N\}\cup\{|u|\leq 2N, |u'|\leq 3N, \tau-\f{\vep}N<\tau_1\leq \tau\}\\
&\cup\{|u|\leq 2N, |u'|\leq 3N, 0\leq \tau_1\leq\tau-\f{\vep}{N}\}.
\end{aligned}$$
For the first and second regions, we have $|u-v|\geq N$ or $|u-u'|\geq N$, such that
either $$|\tilde{k}_r^{ji}(v,u)|\leq e^{-\f{c_1N^2}{2}}\big|\tilde{k}_r^{ji}(v,u)e^{\f{c_1|v-u|^2}{2}}\big|,$$ 
or
$$
|\tilde{k}_r^{lj}(u,u')|\leq e^{-\f{c_1N^2}{2}}\big|\tilde{k}_r^{lj}(u,u')e^{\f{c_1|u-u'|^2}{2}}\big|,
$$
for $i,j,l\in\{A,B\}$.
Therefore, we combine the above argument with \eqref{2.1.11} to bound the integral of $U(\tau_1,u,u';\tau)$ over $\{|u|>2N\}\cup\{|u|\leq 2N,|u'|>3N\}$ by
\begin{align}\label{hinfty1}
&\int_{\{|u|>2N\}\cup\{|u|\leq 2N,|u'|>3N\}}U(\tau_1,u,u';\tau)\dd u'\dd u\dd\tau_1\notag\\
&\leq Cm^{2(\gamma-1)}e^{-\f{c_1N^2}{4}}\|\mathbf{h}\|_{L^{\infty}}\cdot\int_0^{\tau}e^{-\vep^{-1}\hat{\nu}_0(t-\tau_1)}\dd\tau_1\notag\\
&\leq C\vep m^{2(\gamma-1)}e^{-\f{c_0N^2}{4}}\|\mathbf{h}\|_{L^{\infty}}.
\end{align}
For the third region, it is direct to show that
\begin{align}\label{hinfty2}
\int_{\{|u|\leq2N,|u'|\leq 3N,\tau-\f{\vep}N<\tau_1\leq \tau\}}U(\tau_1,u,u';\tau)\dd u'\dd u\dd\tau_1\leq\f{C\vep m^{2(\gamma-1)}}{N}\|\mathbf{h}\|_{L^{\infty}}.
\end{align}
By \eqref{2.1.11}, one has a smooth approximate function $\tilde{k}^{ji}_{r,N}(v,u)$ with compact support, such that
$$
\sup_{|v|\leq 3N}\int_{|u|\leq 3N}|\tilde{k}_{r}^{ji}(v,u)-\tilde{k}_{r,N}^{ji}(v,u)|\dd u\leq N^{-7}, \text{ for }i,j\in\{A,B\}.$$
Using the above approximation and the fact that $$
\begin{aligned}\tilde{k}_{r}^{ji}(v,u)\tilde{k}_{r}^{lj}(u,u')=&\tilde{k}_{r}^{ji}(v,u)
[\tilde{k}_{r}^{lj}(u,u')-\tilde{k}_{r,N}^{lj}(u,u')]\\
&+
\tilde{k}_{r,N}^{lj}(u,u')[\tilde{k}_{r}^{ji}(v,u)-\tilde{k}_{r,N}^{ji}(v,u)]
+\tilde{k}_{r,N}^{ji}(v,u)\tilde{k}_{r,N}^{lj}(u,u'),
\end{aligned}
$$
we have that the integral of $U$ over the last region is bounded as
\begin{align*}
	&\int_{\{|u|\leq 2N, |u'|\leq 3N, 0\leq \tau_1\leq\tau-\f{\vep}{N}\}}U(\tau_1,u,u';\tau)\dd u'\dd u\dd\tau_1\notag\\
&\leq \f{C\vep}{N}\|\mathbf{h}\|_{L^{\infty}}+C_N\int_0^{\tau-\f{\vep}N}e^{-\vep^{-1}\hat{\nu}_0(\tau-\tau_1)}\dd\tau_1\notag\\
&\qquad\qquad\qquad\qquad\times\int_{|u|\leq 2N,|u'|\leq 3N}\big|\tilde{k}_{r,N}^{ji}(v,u)\tilde{k}_{r,N}^{lj}(u,u')h_l(y_1',u')\big|\dd u'\dd u\notag\\
&\leq\f{C\vep}{N}\|\mathbf{h}\|_{L^{\infty}}+C_N\int_0^{\tau-\f{\vep}N}e^{-\vep^{-1}\hat{\nu}_0(\tau-\tau_1)}\dd\tau_1
\bigg(\int_{|u|\leq 2N,|u'|\leq 3N}\big|\tilde{k}_{r,N}^{ji}(v,u)\tilde{k}_{r,N}^{lj}(u,u')\big|^2\dd u'\dd u\bigg)^{\f12}\notag\\
&\qquad\qquad\qquad\qquad\times\bigg(\int_{|u|\leq 2N,|u'|\leq 3N}|\nu_l^{1/2}\psi_l(y_1',u')|\dd u'\dd u\bigg)^{\f12}\notag\\
&\leq \f{C\vep}{N}\|\mathbf{h}\|_{L^{\infty}}+C_N\int_0^{\tau-\f{\vep}N}e^{-\vep^{-1}\hat{\nu}_0(\tau-\tau_1)}\dd\tau_1\bigg(\int_{|u_1|\leq 2N,|u'|\leq 3N}|\nu_l^{1/2}\psi_l(y_1',u')|^2\dd u'\dd u_1\bigg)^{\f12},
\end{align*}
where we have denoted $y_1':=Y^{S}_1(\tau_1)=y_1-(\tau-\tau_1)\hat{u}_1$. Take $\varrho:=1+\f{|\gamma|s^2}{4}$. Then, the Jacobian satisfies 
$$
\big|\f{\dd y_1'}{\dd u_1}\big|=(\varrho+|u|^2)^{\f{|\gamma|}{2}-1}\big|\varrho+|u|^2+|\gamma|(u_1^2-su_1)\big|(\tau-\tau_1)\geq c(1+|u|^2)^{\f{|\gamma|}{2}}(\tau-\tau_1)\geq c\f{\vep}{N}.
$$
Thus, the change of variables $y'_1\rightarrow u_1$ implies
\begin{align}\label{hinfty3}
\int_{|u|\leq 2N, |u'|\leq 3N, 0\leq \tau_1\leq\tau-\f{\vep}{N}}U(\tau_1,u,u';\tau,v)\dd u'\dd u\dd \tau_1\leq \f{C\vep}{N}\|\mathbf{h}\|_{L^{\infty}}+C_N\vep^{1/2}\|\nu^{1/2}\psi\|_{L^2}.
\end{align}
This completes the estimates for integral of $U$ over four regions. Therefore, we collect \eqref{hinfty1}, \eqref{hinfty2} and \eqref{hinfty3} for $|v|\leq N$, together with \eqref{3.2.16-7} for $|v|>N$, to conclude that
\begin{align}\label{3.2.16-8}
|B_2|\leq \f{Cm^{2(\gamma-1)}}{N}\|\mathbf{h}\|_{L^{\infty}}+C_{N,m}\vep^{-1/2}\|\nu^{1/2}\psi\|_{L^2},\ \text{for any $v\in\R^3$}.
\end{align}
It follows from \eqref{3.2.16-4}, \eqref{3.2.16-5}, \eqref{3.2.16-6} and \eqref{3.2.16-8} that
$$\begin{aligned}
\|\mathbf{h}\|_{L^{\infty}}\leq& C\bigg\{m^{2(\gamma-1)}e^{-\vep^{-1}\hat{\nu}_0t}+m^{3+\gamma}+\f{m^{2(\gamma-1)}}{N}\bigg\}\|\mathbf{h}\|_{L^{\infty}}\\
&+C_{N,m}\vep^{-1/2}\|\nu^{1/2}\psi\|_{L^2}
+Cm^{2(\gamma-1)}\vep^{1/2}\|w_{q,\vartheta}\mathbf{r}_2\|_{L^{\infty}}.
\end{aligned}
$$
We further take $Cm^{3+\ga}=\f{1}{4}$ and $t,N>0$ suitably large such that
$$
C\bigg\{m^{2(\gamma-1)}e^{-\vep^{-1}\hat{\nu}_0t}+m^{3+\gamma}+\f{m^{2(\gamma-1)}}{N}\bigg\}\leq \f{1}{2},
$$
so as to obtain that
$$\begin{aligned}
	\|\mathbf{h}\|_{L^{\infty}}\leq C\vep^{-1/2}\|\nu^{1/2}\psi\|_{L^2}
	+C\vep^{1/2}\|w_{q,\vartheta}\mathbf{r}_2\|_{L^{\infty}},
\end{aligned}
$$
which, combined with \eqref{3.2.11}, yields \eqref{3.2.16}. When $0\leq \gamma\leq 1$, \eqref{3.2.16} can be also proved by similar arguments with the mild formulation given by \eqref{3.2.14}. We omit it for brevity. The proof of Lemma \ref{lm3.2.5} is therefore completed.
\end{proof}

Now we are ready to prove Proposition \ref{prop3.2}.
\begin{proof}[ Proof of Proposition \ref{prop3.2}] Define the solution space
\begin{align}
\mathcal{X}:=\{\psi: \nu^{1/2}\psi\in L^2, w_{q,\vartheta}\psi\in L^{\infty},\lim_{y\rightarrow\pm\infty}\psi(y,v)=0,\text{ a.e }v\in\mathbb{R}^3\}\nonumber
\end{align}
with the norm
\begin{align}
\|\psi\|_{\mathcal{X}}:=\vep^{-1/2}\|\nu^{1/2}\psi\|_{L^2}+\|w_{q,\vartheta}\psi\|_{L^{\infty}}.\nonumber
\end{align}
When $\delta=0$, the solution operator $\mathcal{S}_0^{-1}$ to \eqref{3.2.2} has been constructed in Lemma \ref{lm3.2.1}. Next, we construct $\mathcal{S}^{-1}_{\delta}$ for $0<\delta\ll1$. Define the mapping $\mathcal{T}_{\delta}\psi:=\mathcal{S}^{-1}_0(\vep^{-1}\delta\b{K}\psi+\vep^{-1/2}\mathbf{r}_2)$. For any $\psi_1$, $\psi_2\in \mathcal{X}$, it holds from Lemma \ref{lm3.2.2} and Lemma \ref{lm3.2.5} that
$$\begin{aligned}
\|\mathcal{T}_{\delta}(\psi_1-\psi_2)\|_{\mathcal{X}}&=\vep^{-1}\delta\|\mathcal{S}^{-1}_0\b{K}(\psi_1-\psi_2)\|_{\mathcal{X}}\\
&\leq C\delta\vep^{-1}\bigg\{\vep^{-1/2}\|\nu^{1/2}\mathcal{S}^{-1}_0\b{K}(\psi_1-\psi_2)\|_{L^2}+
\|w_{q,\vartheta}\mathcal{S}^{-1}_0\b{K}(\psi_1-\psi_2)\|_{L^{\infty}}\bigg\}\\
&\leq C\delta\vep^{-1}\bigg\{\|\nu^{-1/2}\b{K}(\psi_1-\psi_2)\|_{L^2}+\vep^{1/2}\|\nu^{-1}w_{q,\vartheta}\b{K}(\psi_1-\psi_2)\|_{L^{\infty}}\bigg\}\\
&\leq C\delta\vep^{-1/2}\bigg\{\vep^{-1/2}\|\nu^{1/2}(\psi_1-\psi_2)\|_{L^2}+\|w_{q,\vartheta}(\psi_1-\psi_2)\|_{L^{\infty}}\bigg\}\\
&\leq C\delta\vep^{-1/2}\|\psi_1-\psi_2\|_{\mathcal{X}}.
\end{aligned}
$$
Choose $\delta'=\f{\vep^{1/2}}{2C}$ to conclude that $\mathcal{T}_{\delta}$ is a contraction mapping for $0<\delta\leq \delta'$ such that it has a fixed point $\psi_{\delta}$ which solves the equation \eqref{3.2.2}. Thus the construction of $\mathcal{S}_{\delta}^{-1}$, for $0\leq\delta\leq \delta'$, is completed. Next, we define $\mathcal{T}_{\delta'+\delta}\psi:=\mathcal{S}^{-1}_{\delta'}(\vep^{-1}\delta\b{K}\psi+\vep^{-1/2}\mathbf{r}_2)$. Notice that by our uniform estimates \eqref{3.2.11} and \eqref{3.2.16}, the upper bounds on the norms of solution operator $\mathcal{S}_{\delta'}^{-1}$ are independent of $\delta'$. Therefore, by utilizing the same argument, it can be shown that $\mathcal{T}_{\delta'+\delta}$ is also a contraction mapping in $\mathcal{X}$ for $0<\delta\leq \delta'$, so that it has a fixed point in $\mathcal{X}$. This yields the existence of the solution operator $\mathcal{S}_{2\delta'}^{-1}$. Following the procedure step by step, the solution operator $\mathcal{S}^{-1}_1$ is constructed. Finally, the continuity follows directly from $L^{\infty}$-convergence. Therefore, the proof of Proposition \ref{prop3.2} is completed. \end{proof}
\subsection{Nonlinear problem}
This part is devoted to construct solutions to the nonlinear coupled system \eqref{3.1.13} and \eqref{3.1.15}. We let $y\rightarrow\infty$ in \eqref{3.1.13} and \eqref{3.1.15} to obtain
\begin{align}\label{zinf}
    \lambda' z_\infty+\langle\mu_{\vep},(I-P)&(z_\infty\varphi_{\vep}+\vep\psi_\infty)\rangle\notag\\&-\langle\varphi_{\vep},\mathbf{\Gamma}(z_\infty\phi_{\vep}+\vep\psi_\infty,\psi_\infty)
+\mG(\psi_\infty,z_\infty\phi_{\vep})\rangle=0.
\end{align}
Then  \eqref{3.1.13}$-p^2(y)\times$\eqref{zinf} gives a new expression of the nonlinear term $r_1=J_1+N_1$ that
$$
\begin{aligned}
J_1=&p(y)(1-p(y))\big[\lambda'z_\infty+{\langle\mu_{\vep},z_{\infty}\varphi_{\vep}\rangle}\big]\\
&+p^2(y)(1-p(y))
\langle\varphi_{\vep},\mG(z_{\infty}\phi_{\vep}+\vep p\psi_{\infty},\psi_\infty)+\vep\mG(\psi_\infty,\psi_{\infty})+\mG(\psi_\infty,z_\infty \phi_\vep)\rangle
\end{aligned}
$$
and $$
\begin{aligned}
N_1(z_1,&\psi_1)=-\vep\zeta_{\vep}z_1^2+\vep\langle\mu_{\vep},(I-P)(z_1\phi_{\vep}+\vep^{1/2}\psi_1)\rangle\\
&-
\vep^{1/2}\bigg\langle\varphi_{\vep},\mG(\mathbf{f},\psi_1)+
\mG(\psi_1,z_0\phi_{\vep})+\vep^{1/2}\mG(z_1\phi_{\vep}+\vep^{1/2}\psi_1,\psi_0)+\vep^{1/2}\mG(\psi,z_1\phi_{\vep})\bigg\rangle.
\end{aligned}
$$
Here, $p(y)$ is the tanh-function defined in \eqref{p}.

\begin{proposition}\label{prop3.3}
Let $-3<\gamma\leq1,$ $\theta=\f{3}{2-\gamma}$, $0<q\leq \hat{q}_1$ for $\hat{q}_1$ given in Proposition \ref{prop3.2} and $\vartheta>2$.  There exist positive constants $\vep_3$ and $\varpi>0$, such that for any $0<\vep\leq \vep_3$ and $z_{1,0}\in \mathbb{R}$, the remainder system \eqref{3.1.13} and \eqref{3.1.15} with $z_{1}(0)=z_{1,0}$ admits a unique solution $[z_1(y),\psi_1(y,v)]^T$, which satisfies
\begin{align}\label{3.3.1}
|e^{\varpi|\cdot|^{\theta}}z_1|_{L^2_y}+|z_1|_{L^{\infty}_y}+\|\nu^{1/2}e^{\varpi|\cdot|^{\theta}}\psi_1\|_{L^2}+\|w_{q,\vartheta}\psi_1\|_{L^{\infty}}\leq C.
\end{align}
Here, the constant $C>0$ is independent of $\vep$. Moreover, $z_1$ is continuous and $\psi_1$ is continuous away from $\mathbb{R}\times\{v_1=s\}$.
\end{proposition}

\begin{proof}
Without loss of generality, we assume that $z_{1}(0)=z_{1,0}=0.$ The solution is to be constructed via the following iteration problems:
\begin{equation}\label{3.3.2}
\left\{
\begin{aligned}
&\pa_yz_1^{n+1}-\lambda z_1^{n+1}+2\zeta_{\vep}z_0z_1^{n+1}=J_1+N_1(z_1^n,\psi_1^n),\\
&(v_1-s)\pa_y\psi_1^{n+1}+\vep^{-1}\mathbf{H}\psi_1^{n+1}=\vep^{-1/2}[\mathbf{J}_2+\mathbf{N}_2(z_1^{n},\psi_1^n)],\\
&z_1(0)=0,\quad\lim_{y\rightarrow\pm\infty}z_1^{n+1}(y)=0,\quad\lim_{y\rightarrow\pm\infty}\psi_1^{n+1}(y,v)=0,\\
\end{aligned}\right.
\end{equation}
with $z_{1}^0=0$ and $\psi_1^0(y,v)=0$. Now, we estimate terms on the right hand side of the first two equations in \eqref{3.3.2}. By using the fact
\begin{equation}\label{3.3.3}\left\{
\begin{aligned}
&\big|\nu_j^{-1/2}\Gamma^{ij}(f_i,f_j)\big|_{L^2_v}\leq C|w_{q,\vartheta}f_i|_{L^{\infty}_v}\cdot|\nu_j^{1/2}f_j|_{L^2_v},\\
&\big|\nu_j^{-1}w_{q,\vartheta}\Gamma^{ij}(f_i,f_j)\big|_{L^{\infty}_v}\leq C|w_{q,\vartheta}f_i|_{L^{\infty}_v}\cdot|w_{q,\vartheta}f_j|_{L^\infty_v},
\end{aligned}\right.
\end{equation}
and exponential decay of $p(y)$ in \eqref{p}, we have
\begin{align}\label{3.3.4}
|W_{\varpi}J_1|_{L^2_y}+|J_1|_{L^{\infty}_y}
+\|\nu^{-1/2}W_{\varpi}\mathbf{J}_2\|_{L^2}+\|\nu^{-1}w_{q,\vartheta}\mathbf{J}_2\|_{L^{\infty}}\leq C.
\end{align}
Here, we have denoted the spatial weight function $W_{\varpi}(y):=e^{\varpi(|y|^2+1)^{\f{\theta}{2}}}$ with $0<\varpi<\f{\lambda}{4}$ for $\theta=1$, and $\varpi>0$ for $0<\theta<1$. Similarly, for $N_1$, it holds that
\begin{align}\label{3.3.5}
|W_{\varpi}N_1(z_1,\psi_1)|_{L^2_y}\lesssim\vep^{1/2}\{1+|z_1|_{L^{\infty}_y}+\|w_{q,\vartheta}\psi_1\|_{L^{\infty}}\}\cdot
\{|W_{\varpi}z_1|_{L^2_y}+\|\nu^{1/2}W_{\varpi}\psi_1\|_{L^{2}}\},
\end{align}
and
\begin{align}\label{3.3.6}
|N_1(z_1,\psi_1)|_{L^{\infty}_y}\lesssim\vep^{1/2}\{1+|z_1|_{L^{\infty}_y}+\|w_{q,\vartheta}\psi_1\|_{L^{\infty}}\}\cdot
\{|z_1|_{L^\infty_y}+\|w_{q,\vartheta}\psi_1\|_{L^{\infty}}\}.
\end{align}
As for $N_2$, we have
\begin{align}\label{3.3.7}
\|&\nu^{-1/2}W_{\varpi}\mathbf{N}_2(z_1,\psi_1)\|_{L^2}\nonumber\\
&\quad\lesssim |W_{\varpi}z_1|_{L^2_y}+\vep^{1/2}\{1+|z_1|_{L^{\infty}_{y}}+\|w_{q,\vartheta}\psi_1\|_{L^{\infty}}\}\cdot
\{|W_{\varpi}z_1|_{L^2_{y}}+\|\nu^{1/2}W_{\varpi}\psi_1\|_{L^{2}}\},
\end{align}
and
\begin{align}\label{3.3.8}
\|\nu^{-1}&w_{q,\vartheta}\mathbf{N}_2(z_1,\psi_1)\|_{L^\infty}\lesssim|z_1|_{L^\infty_{y}}
+\vep^{1/2}\{1+|z_1|_{L^{\infty}_{y}}+\|w_{q,\vartheta}\psi_1\|_{L^{\infty}}\}\cdot
\{|z_1|_{L^\infty_{y}}+\|w_{q,\vartheta}\psi_1\|_{L^{\infty}}\}.
\end{align}
Combine \eqref{3.3.4}, \eqref{3.3.5}, \eqref{3.3.6}, \eqref{3.3.7} and \eqref{3.3.8} with Proposition \ref{prop3.1} and Proposition \ref{prop3.2} to get
\begin{align}\label{3.3.9}
|W_{\varpi}z_1^{n+1}|_{L^2_{y}}+|z_1^{n+1}|_{L^{\infty}_{y}}\leq &\hat{C}_1+C\vep^{1/2}\{1+|z_1^n|_{L^{\infty}_y}+\|w_{q,\vartheta}\psi_1^n\|_{L^{\infty}}\}\nonumber\\
&\times\{|W_{\varpi}z_1^n|_{L^2_{y}}+|z_1^n|_{L^\infty_{y}}
+\|\nu^{1/2}W_{\varpi}\psi_1^n\|_{L^{2}}+\|w_{q,\vartheta}\psi_1^n\|_{L^{\infty}}\},
\end{align}
and
\begin{align}\label{3.3.10}
\|w_{q,\vartheta}\psi_1^{n+1}\|_{L^{\infty}}\leq& \hat{C}_1+\tilde{C}_1|W_{\varpi}z^n_1|_{L^2_y}+C\vep^{1/2}\{1+|z_1^{n}|_{L^\infty_{y}}+\|w_{q,\vartheta}\psi_1^n\|_{L^{\infty}}\}\nonumber\\
&\times
\{|z_1^n|_{L^{\infty}_y}+\|w_{q,\vartheta}\psi_1^n\|_{L^{\infty}}+|W_{\varpi}z_1^n|_{L^2_y}+\|\nu^{1/2}W_{\varpi}\psi_1^n\|_{L^2}\}.
\end{align}
Here, the positive constants $\hat{C}_1$, $\tilde{C}_1$ and $C$ are independent of $\vep$.
To close our estimate, it suffices to estimate the spatial-weighted $L^2$-norm $\|\nu^{1/2}W_{\varpi}\psi_1^{n+1}\|_{L^2}$. This can be achieved through an interplay with $L^{\infty}$-norm of higher-moments $\|w_{q,\vartheta}\psi_{1}^{n+1}\|$. In fact, we can show that
\begin{align}
\|\nu^{1/2}W_{\varpi}\psi_1^{n+1}\|_{L^{2}}\leq& \hat{C}_2+C\vep^{1/2}\|w_{q,\vartheta}\psi_1^{n+1}\|_{L^{\infty}}\label{3.3.11}\nonumber\\
&+C\vep^{1/2}\{1+|z_1^{n}|_{L^\infty_{y}}+\|w_{q,\vartheta}\psi_1^n\|_{L^{\infty}}\}\cdot
\{|W_{\varpi}z_1^n|_{L^2_y}+\|\nu^{1/2}W_{\varpi}\psi_1^n\|_{L^2}\},
\end{align}
for some positive constant $\hat{C}_2$. Notice that it would be very difficult to get \eqref{3.3.11} directly by energy method, since the integrating by parts is not valid due to the lack of asymptotic behavior $\lim_{y\rightarrow \pm\infty} W_{\varpi}\psi_1^{n+1}=0$. Thus, we need a pointwise control of $W_{\varpi}\psi_1^{n+1}$. To do this, fixing $y<0$, then
multiplying $W^2_{\varpi}\psi_1^{n+1}$ to $\eqref{3.3.2}_2$ and integrating over $[y,0]\times\R^3$, we have
\begin{align}\label{3.3.12}
&\underbrace{\vep^{-1}\int_y^0\langle\mathbf{H}\psi_1^{n+1},W_{\varpi}^2\psi_1^{n+1}\rangle\dd y_1}_{I_5}\nonumber\\
&\quad=\underbrace{-\int^0_y\langle(v_1-s)\ \pa_y\psi_1^{n+1},W_{\varpi}^2\psi_1^{n+1}\rangle\dd y_1}_{I_6}+\underbrace{\vep^{-1/2}\int_y^0\langle(\mathbf{J}_2+\mathbf{N}_2),W_{\varpi}^2\psi_1^{n+1}\rangle\dd y_1}_{I_7}.
\end{align}
Using the positivity of $\mathbf{H}$ in \eqref{3.1.16}, one gets
\begin{align}\label{I5}
I_5\geq c_1\vep^{-1}\int_y^0|W_{\varpi}\nu^{1/2}\psi_1^{n+1}|_{{L^2_v}}^2\dd y_1.
\end{align}
By Cauchy-Schwarz inequality, for any $0< \eta<1$, it holds that
\begin{align}\label{I7}
I_7\leq& \eta\vep^{-1}\int_y^0|W_{\varpi}\nu^{1/2}\psi_1^{n+1}|_{{L^2_v}}^2\dd y_1\notag\\
&\quad+C_{\eta}\int_y^0|W_{\varpi}\nu^{-1/2}\mathbf{J}_2|_{{L^2_v}}^2+|W_{\varpi}\nu^{-1/2}\mathbf{N}_2|_{{L^2_v}}^2\dd y_1\notag\\
\leq&\eta\vep^{-1}\int_y^0|W_{\varpi}\nu^{1/2}\psi_1^{n+1}|_{{L^2_v}}^2\dd y_1+C_{\eta}\|W_{\varpi}\nu^{-1/2}\mathbf{J}_2\|_{L^2}^2+C_{\eta}\|W_{\varpi}\nu^{-1/2}\mathbf{N}_2\|_{L^2}^2.
\end{align}
We further apply integrating by parts to obtain
\begin{align}\label{3.3.13}
I_6=\f{\langle W_{\varpi}^2(v_1-s)\psi_1^{n+1},\psi_1^{n+1}\rangle}{2}\bigg|_{y_1=0}^{y_1=y}+\int_y^0W_{\varpi}W_{\varpi}'\langle(v_1-s)\psi_1^{n+1},\psi_1^{n+1}\rangle\dd y_1.
\end{align}
Noticing for our choice $\vartheta>2$, the first term on the right hand side above can be bounded by
\begin{align}\label{I61}
\f{|\langle W_{\varpi}^2(v_1-s)\psi_1^{n+1},\psi_1^{n+1}\rangle|}{2}\bigg|_{y_1=0}\leq C\sup_{y,v}\big|(1+|v|^2)^{\vartheta}\psi_1^{n+1}\big|^2\leq C\|w_{q,\vartheta}\psi_1^{n+1}\|_{L^{\infty}}^2.
\end{align}
Next, to control $\f{\langle(v_1-s)\psi_1^{n+1},\psi_1^{n+1}\rangle}{2}\bigg|_{y_1=y}$ pointwisely, multiplying $\eqref{3.3.2}_2$ by $\psi_1^{n+1}$ and then integrating over $(-\infty,y]$, we have
\begin{align}
&\f{\langle(v_1-s)\psi_1^{n+1},\psi_1^{n+1}\rangle}{2}\bigg|_{y_1=y}+\vep^{-1}\int_{-\infty}^{y}
\langle\mathbf{H}\psi_1^{n+1},\psi_1^{n+1}\rangle\dd y_1=\vep^{-1/2}\int_{-\infty}^{y}\langle\mathbf{J}_2+\mathbf{N}_2
,\psi_1^{n+1}\rangle\dd y_1.\nonumber
\end{align}
Then, by Cauchy-Schwarz inequality, it holds, for any $\eta'>0$, that
\begin{align}
&\f{\langle(v_1-s)\psi_1^{n+1},\psi_1^{n+1}\rangle}{2}\bigg|_{y_1=y}+\vep^{-1}\int_{-\infty}^{y}
\langle\mathbf{H}\psi_1^{n+1},\psi_1^{n+1}\rangle\dd y_1\nonumber\\
&\leq \eta'\vep^{-1}\int_{-\infty}^{y}|\nu^{1/2}\psi^{n+1}_1|_{{L^2_v}}^2\dd y_1+C_{\eta'}\int_{-\infty}^y|\nu^{-1/2}\mathbf{J}_2|_{{L^2_v}}^2+|\nu^{-1/2}\mathbf{N}_2|_{{L^2_v}}^2\dd y_1.\nonumber
\end{align}
Use the positivity of $\mathbf{H}$ \eqref{3.1.16} again and take $\eta'>0$ suitably small to get
\begin{align}\label{3.3.14}
\f{\langle(v_1-s)\psi_1^{n+1},\psi_1^{n+1}\rangle}{2}\bigg|_{y_1=y}\leq C\int_{-\infty}^y\big\{|\nu^{-1/2}\mathbf{J}_2|_{{L^2_v}}^2+|\nu^{-1/2}\mathbf{N}_2|_{{L^2_v}}^2\big\}\dd y_1.
\end{align}
Then, multiplying $W^2_{\varpi}(y)$ to \eqref{3.3.14}, together with the fact that $W^2_{\varpi}(y)\leq W^2_{\varpi}(y_1)$ for $y_1\leq y<0$, one has
\begin{align}
\f{\langle(v_1-s)\psi_1^{n+1},W^2_{\varpi}\psi_1^{n+1}\rangle}{2}\bigg|_{y_1=y}&\leq CW^{2}_{\varpi}\int_{-\infty}^{y}|\nu^{-1/2}\mathbf{J}_2|_{{L^2_v}}^2+|\nu^{-1/2}\mathbf{N}_2|_{{L^2_v}}^2\dd y_1\nonumber\\
&\leq C\|\nu^{-1/2}W_{\varpi}\mathbf{J}_2\|_{L^2}^2+C\|\nu^{-1/2}W_{\varpi}\mathbf{N}_2\|_{L^2}^2.\notag
\end{align}
It remains to estimate the second term on the right hand side of \eqref{3.3.13}. First, it is straightforward to show that
\begin{align}\label{I621}
\bigg|\int_y^0W_{\varpi}W_{\varpi}'\langle(v_1-s)\psi_1^{n+1},\psi_1^{n+1}\rangle\dd y_1\bigg|\leq C\int_y^0W_{\varpi}^2(1+|y_1|^2)^{\f{\theta-1}{2}}\int_{\mathbb{R}^3}(1+|v|)|\psi^{n+1}|^2\dd v.
\end{align}
Then, we split the integral domain over $v$ into two parts:
$$\{|v|>(1+|y_1|^2)^{\f{\theta}{4}}\}\cup\{|v|\leq(1+|y_1|^2)^{\f{\theta}{4}}\}.
$$
For the first region, we have, for $0<\varpi\leq \f{q}{8} $, that
$$
e^{-\f{q}{8}|v|^2}\leq e^{-\f{q}{8}(1+|y_1|^2)^{\f\theta2}}\leq W^{-1}_{\varpi}(y_1)\leq W^{-1}_{\varpi}(y),
$$
which leads to
$$\begin{aligned}
\int_{|v|>(1+|y_1|^2)^{\f{\theta}{4}}}(1+|v|)|\psi_1^{n+1}|^2\dd v&\leq W^{-4}_{\varpi}\int_{|v|>(1+|y_1|^2)^{\f{\theta}{4}}}(1+|v|)e^{\f{q|v|^2}{2}}|\psi_1^{n+1}|^2\dd v\\
&\leq CW^{-4}_{\varpi}\|w_{q,\vartheta}\psi_1^{n+1}\|_{L^{\infty}}^2.
\end{aligned}
$$
For the second region, a direct computation shows that
$$
\int_{|v|\leq(1+|y_1|^2)^{\f{\theta}{4}}}(1+|v|)|\psi^{n+1}_1|^2\dd v\leq C(1+|y_1|^2)^{\f{\theta(1-\gamma)}{4}}|\nu^{1/2}\psi_1^{n+1}|_{{L^2_v}}^2.
$$
Collect these two bounds together with \eqref{I621} to obtain, for $\theta=\f{2}{3-\gamma}$, that
\begin{align*}
&\bigg|\int_y^0W_{\varpi}W_{\varpi}'\langle(v_1-s)\psi_1^{n+1},\psi_1^{n+1}\rangle\dd y_1\bigg|\notag\\
&\qquad\leq C\|w_{q,\vartheta}\psi_1^{n+1}\|_{L^\infty}^2\cdot\int_y^0W_{\varpi}^{-2}(1+|y_1|^2)^{\f{\theta-1}{2}}\dd y_1\notag\\
&\qquad\quad+C\int_y^0W^2_{\varpi}(1+|y_1|^2)^{\f{\theta-1}{2}+\f{\theta(1-\gamma)}{4}}|\nu^{1/2}\psi_1^{n+1}|_{{L^2_v}}^2\dd y_1\notag\\
&\qquad\leq C\|w_{q,\vartheta}\psi_1^{n+1}\|_{L^\infty}^2+C\|W_{\varpi}\nu^{1/2}\psi_1^{n+1}\|_{L^2}^2,
\end{align*}
which, together with \eqref{I61}, yields
\begin{align}\label{I6}
	&I_6\leq C\|w_{q,\vartheta}\psi_1^{n+1}\|_{L^\infty}^2+C\|W_{\varpi}\nu^{1/2}\psi_1^{n+1}\|_{L^2}^2.
\end{align}
It follows from \eqref{3.3.12}, \eqref{I5}, \eqref{I7} and \eqref{I6} with suitably small $\eta>0$ that
\begin{align}\label{3.3.16}
&\vep^{-1}\int_y^0|W_{\varpi}\nu^{1/2}\psi_1^{n+1}|_{{L^2_v}}^2\dd y_1\nonumber\\
&\leq \eta\vep^{-1}\int_y^0|W_{\varpi}\nu^{1/2}\psi_1^{n+1}|_{{L^2_v}}^2\dd y_1+C_{\eta}\|w_{q,\vartheta}\psi_1^{n+1}\|_{L^{\infty}}^2+C_{\eta}\|\nu^{1/2}W_{\varpi}\psi_1^{n+1}\|_{L^2}^2
\nonumber\\
&\qquad+C_{\eta}\{\|\nu^{-1/2}W_{\varpi}\mathbf{J}_2\|_{L^2}^2+\|\nu^{-1/2}W_{\varpi}\mathbf{N}_2\|_{L^2}^2\}\nonumber\\
&\leq C\|w_{q,\vartheta}\psi_1^{n+1}\|_{L^{\infty}}^2+C\|\nu^{1/2}W_{\varpi}\psi_1^{n+1}\|_{L^2}^2\notag\\
&\qquad
+C\|\nu^{-1/2}W_{\varpi}\mathbf{J}_2\|_{L^2}^2+C\|\nu^{-1/2}W_{\varpi}\mathbf{N}_2\|_{L^2}^2.
\end{align}
Note that our constant $C>0$ is independent of $y$. Therefore, letting $y\rightarrow-\infty$ in \eqref{3.3.16}, we have
$$
\begin{aligned}
&\vep^{-1}\int_{-\infty}^0|W_{\varpi}\nu^{1/2}\psi_1^{n+1}|_{{L^2_v}}^2\dd y_1\\
&\leq C\|w_{q,\vartheta}\psi_1^{n+1}\|_{L^{\infty}}^2+C\|\nu^{1/2}W_{\varpi}\psi_1^{n+1}\|_{L^2}^2
+C\|\nu^{-1/2}W_{\varpi}\mathbf{J}_2\|_{L^2}^2+C\|\nu^{-1/2}W_{\varpi}\mathbf{N}_2\|_{L^2}^2\\
&\leq C.
\end{aligned}
$$
In a similar way, one gets that
$$
\begin{aligned}
\vep^{-1}&\int_{0}^\infty|W_{\varpi}\nu^{1/2}\psi_1^{n+1}|_{{L^2_v}}^2\dd y_1\\
&\leq C\|w_{q,\vartheta}\psi_1^{n+1}\|_{L^{\infty}}^2+C\|\nu^{1/2}W_{\varpi}\psi_1^{n+1}\|_{L^2}^2
+C\|\nu^{-1/2}W_{\varpi}\mathbf{J}_2\|_{L^2}^2+C\|\nu^{-1/2}W_{\varpi}\mathbf{N}_2\|_{L^2}^2.
\end{aligned}
$$
Combining the two estimates above with \eqref{3.3.4} and \eqref{3.3.7}, we have
$$
\begin{aligned}
\|W_{\varpi}\nu^{1/2}\psi_1^{n+1}\|_{L^2}&\leq C+C\vep^{1/2}\|w_{q,\vartheta}\psi_1^{n+1}\|_{L^{\infty}}+C\vep^{1/2}\|\nu^{1/2}W_{\varpi}\psi_1^{n+1}\|_{L^2}\\
&\qquad+C\vep^{1/2}\|\nu^{-1/2}W_{\varpi}\mathbf{J}_2\|_{L^2}+C\vep^{1/2}\|\nu^{-1/2}W_{\varpi}\mathbf{N}_2\|_{L^2}^2\\
&\leq C+C\vep^{1/2}\|w_{q,\vartheta}\psi_1^{n+1}\|_{L^{\infty}}+C\vep^{1/2}\|\nu^{1/2}W_{\varpi}\psi_1^{n+1}\|_{L^2}\\
&\qquad+C\vep^{1/2}\{1+|z_1^{n}|_{L^\infty_{y}}+\|w_{q,\vartheta}\psi_1^n\|_{L^{\infty}}\}\cdot
\{|W_{\varpi}z_1^n|_{L^2_y}+\|\nu^{1/2}W_{\varpi}\psi_1^n\|_{L^2}\}.
\end{aligned}
$$
Then \eqref{3.3.11} follows by taking $\vep>0$ suitably small. Let $\hat{C}_3:=\max\{\hat{C}_1,\hat{C}_2,1\}$ and $\hat{C}_4:=\max\{\tilde{C}_1,1\}$. Introduce a norm
\begin{align}\label{def.norm.Y}
\|[z_1,\psi_1]\|_{\mathcal{Y}}:=2\hat{C}_4\{\|W_{\varpi}z_1\|_{L^2_{y}}+\|z_1\|_{L^{\infty}_{y}}\}
+\|\nu^{1/2}W_{\varpi}\psi_1\|_{L^2}+\|w_{q,\vartheta}\psi_1\|_{L^{\infty}}.
\end{align}
We will show inductively that
\begin{align}\label{3.3.17}
\|[z_1^n,\psi_1^n]\|_{\mathcal{Y}}\leq 8(\hat{C}_4+1)\hat{C}_3.
\end{align}
Notice that \eqref{3.3.17} holds automatically when $n=0$. Assume it holds for $n=k$. Then, by \eqref{3.3.9}, \eqref{3.3.10} and \eqref{3.3.11}, we have
\begin{align*}
\|[z_1^{k+1},\psi_1^{k+1}]\|_{\mathcal{Y}}&\leq 2\hat{C}_3(\hat{C}_4+1)+\f12\|[z_1^k,\psi^k_1]\|_{\mathcal{Y}}+C\vep^{1/2}\|[z_1^{k+1},\psi^{k+1}_1]\|_{\mathcal{Y}}\nonumber
\\
&\quad+2C(\hat{C}_4+1)\vep^{1/2}\|[z_1^k,\psi^k_1]\|_{\mathcal{Y}}^2\nonumber\\
&\leq6\hat{C}_3(\hat{C}_4+1)+128C\hat{C}_3^2(\hat{C}_4+1)^3\vep^{1/2}+C\vep^{1/2}\|[z_1^{k+1},\psi^{k+1}_1]\|_{\mathcal{Y}},
\end{align*}
which yields \eqref{3.3.17}
by taking $\vep>0$ so small that
$C\vep^{1/2}\leq \f{1}{8}$ and $128C\hat{C}_3(\hat{C}_4+1)^2\vep^{1/2}\leq 1.$ To prove the convergence of $[z_1^n,\psi_1^n]$, we consider the difference $[z_1^{n+1}-z_1^n,\psi_1^{n+1}-\psi_1^n]:=[\b{z}_1^{n+1},\b{\psi}_{1}^{n+1}]$. Notice that it satisfies the following system:
$$\left\{
\begin{aligned}
&\pa_y\b{z}_1^{n+1}-\lambda\b{z}_{1}^{n+1}+2\zeta_{\vep}z_0\b{z}_{1}^{n+1}=N_1(z_1^{n},\psi_1^{n})
-N_1(z_1^{n-1},\psi_1^{n-1}),\\
&(v_1-s)\pa_y\b{\psi}_1^{n+1}+\vep^{-1}\mathbf{H}\b{\psi}_1^{n+1}=\vep^{-1/2}[\mathbf{N}_2(z_1^{n},\psi_1^{n})
-\mathbf{N}_2(z_1^{n-1},\psi_1^{n-1})],\\
&\lim_{y\rightarrow\pm\infty}\b{z}_1^{n+1}(y)=0,\b{z}_1^{n+1}(0)=0,\lim_{y\rightarrow\pm\infty}\b{\psi}_{1}^{n+1}(y,v)=0.
\end{aligned}
\right.$$
Using \eqref{3.3.3}, we get
$$
\begin{aligned}
&|W_{\varpi}[N_1(z_1^n,\psi_1^n)-N_1(z_1^{n-1},\psi_1^{n-1})]|_{L^2_{y}}\\
&\quad\leq C\vep^{1/2}\{1+
\|[z_1^{n},\psi_1^n]\|_{\mathcal{Y}}+\|[z_1^{n-1},\psi_1^{n-1}]\|_{\mathcal{Y}}\}
\{|W_{\varpi}\b{z}^n_1|_{L^2_{y}}+\|\nu^{1/2}W_{\varpi}\b{\psi}^n_1\|_{L^{2}}\},\\
&|[N_1(z_1^n,\psi^n_1)-N_1(z_1^{n-1},\psi_1^{n-1})]|_{L^\infty_{y}}\\
&\quad\leq C\vep^{1/2}\{1+
\|[z_1^{n},\psi_1^n]\|_{\mathcal{Y}}+\|[z_1^{n-1},\psi_1^{n-1}]\|_{\mathcal{Y}}\}
\{|\b{z}^n_1|_{L^\infty_{y}}+\|W_{q,\vartheta}\b{\psi}^n_1\|_{L^{\infty}}\},\\
&\|\nu^{-1/2}W_{\varpi}[\mathbf{N}_2(z_1^n,\psi_1^n)-\mathbf{N}_2(z_1^{n-1},\psi_1^{n-1})]\|_{L^2}\\
&\quad\leq C|W_{\varpi}\b{z}^n_1|_{L^2_{y}}+C\vep^{1/2}\{1+\|[z_1^n,\psi_1^n]\|_{\mathcal{Y}}+\|[z_1^{n-1},\psi_1^{n-1}]\|_{\mathcal{Y}}\}
\{|W_{\varpi}\b{z}^n_1|_{L^2_{y}}+\|\nu^{1/2}W_{\varpi}\b{\psi}^n_1\|_{L^{2}}\},
\end{aligned}
$$
and
$$
\begin{aligned}
	&\|\nu^{-1}w_{q,\vartheta}[\mathbf{N}_2(z_1^n,\psi_1^n)-\mathbf{N}_2(z_1^{n-1},\psi^{n-1}_1)]\|_{L^\infty}\\
	&\quad\leq C|\b{z}^n_1|_{L^\infty_{y}}+C\vep^{1/2}\{1+\|[z_1^n,\psi_1^n]\|_{\mathcal{Y}}+\|[z_1^{n-1},\psi_1^{n-1}]\|_{\mathcal{Y}}\}
	\{|\b{z}^n_1|_{L^\infty_{y}}+\|w_{q,\vartheta}\b{\psi}^n_1\|_{L^{\infty}}\}.
\end{aligned}
$$
Utilizing Propositions \ref{prop3.1} and \ref{prop3.2} to $[\b{z}_1^{n+1},\b{\psi}_1^{n+1}]$, together with the four estimates above, we have
\begin{align}\label{3.3.18}
|W_{\varpi}\b{z}_1^{n+1}|_{L^2_{y}}+|&\b{z}_1^{n+1}|_{L^{\infty}_{y}}\leq C\vep^{1/2}\{1+\|[z_1^n,\psi_1^n]\|_{\mathcal{Y}}+\|[z_1^{n-1},\psi_1^{n-1}\|_{\mathcal{Y}}\}\nonumber\\
&\times\{|W_{\varpi}\b{z}_1^n|_{L^2_{y}}+|\b{z}_1^n|_{L^\infty_{y}}
+\|\nu^{1/2}W_{\varpi}\b{\psi}_1^n\|_{L^{2}}+\|w_{q,\vartheta}\b{\psi}_1^n\|_{L^{\infty}}\},
\end{align}
and
\begin{align}\label{3.3.19}
\|w_{q,\vartheta}\b{\psi}_1^{n+1}\|_{L^{\infty}}\leq & \hat{C}_5|W_{\varpi}\b{z}_1^{n}|_{L^2_{y}}+C\vep^{1/2}\{1+\|[z_1^n,\psi_1^n]\|_{\mathcal{Y}}+\|[z_1^{n-1},\psi_1^{n-1}]\|_{\mathcal{Y}}\}\nonumber\\
&\times\{|W_{\varpi}\b{z}_1^n|_{L^2_{y}}+|\b{z}_1^n|_{L^\infty_{y}}
+\|\nu^{1/2}W_{\varpi}\b{\psi}_1^n\|_{L^{2}}+\|w_{q,\vartheta}\b{\psi}_1^n\|_{L^{\infty}}\},
\end{align}
where the constant $\hat{C}_5>1$ is independent of $\vep$. Moreover, by the similar argument to how we get \eqref{3.3.11}, one can also show that
\begin{align}\label{3.3.20}
\|\nu^{1/2}W_{\varpi}\b{\psi}_1^{n+1}\|_{L^{2}}\leq& C\vep^{1/2}\|w_{q,\vartheta}\b{\psi}_1^{n+1}\|_{L^{\infty}}
+C\vep^{1/2}\{1+\|[z_1^n,\psi_1^n]\|_{\mathcal{Y}}+\|[z_1^{n-1},\psi_1^{n-1}]\|_{\mathcal{Y}}\}\nonumber\\
&\times\{|W_{\varpi}\b{z}_1^n|_{L^2_{y}}+\|\nu^{1/2}W_{\varpi}\b{\psi}_1^n\|_{L^{2}}\}.
\end{align}
Let $g_n:=2\hat{C}_5\{|W_{\varpi}\b{z}_1^{n}|_{L^2_{y}}+|\b{z}_1^{n}|_{L^\infty_{y}}\}
+\|\nu^{1/2}W_{\varpi}\b{\psi}_1^{n}\|_{L^{2}}+\|w_{q,\vartheta}\b{\psi}_1^{n}\|_{L^{\infty}}.$ Then, it holds by $2\hat{C}_5\times\eqref{3.3.18}+\eqref{3.3.19}+\eqref{3.3.20}$, together with \eqref{3.3.17}, that
$$
\begin{aligned}
g_{n+1}\leq \f{1}{2}g_n+C\vep^{1/2}g_{n+1}+2(\hat{C}_5+1)\{1+16\hat{C}_3(\hat{C}_4+1)\}\vep^{1/2}g_n.
\end{aligned}
$$
We further choose $\vep>0$ small enough such that
$$
C\vep^{1/2}\leq \f{1}{16}\text{ and }2(\hat{C}_5+1)\{1+16\hat{C}_3(\hat{C}_4+1)\}\vep^{1/2}\leq \f{1}{4},
$$
so as to get
$$
\begin{aligned}
	g_{n+1}\leq \f{4}{5}g_n,
\end{aligned}
$$ which implies that $[z_1^n,\psi_1^n]$ is a Cauchy sequence. Then, we can construct the solution to \eqref{3.1.13} and \eqref{3.1.15} by taking $[z_1,\psi_1]:=\lim_{n\rightarrow\infty}[z_1^n,\psi_1^n]$ which satisfies \eqref{3.3.1} according to \eqref{3.3.17} with \eqref{def.norm.Y}. Finally, the continuity follows from the $L^{\infty}$-convergence. Therefore, the proof of Proposition \ref{prop3.3} is completed.
\end{proof}
Higher spatial regularity is given by the following Corollary.
\begin{corollary}\label{cor3.4}
Let $N>0$ be an integer, $-3<\gamma\leq 1$ and $\theta=\f{2}{3-\gamma}$. Under the same assumption as in Proposition \ref{prop3.3}, there exist positive constants $C$ and $\varpi$ independent of $\vep$, such that
\begin{align}\label{3.3.21}
\sum_{k=0}^N\sup_{y}\big|e^{\varpi|y|^{\theta}}\pa^k_yz_1\big|+\sup_y\big|e^{\varpi|y|^{\theta}}|\nu^{1/2}\pa^k_y\psi_1|_{{L^2_v}}\big|+\|w_{q,\vartheta}\pa^k_y\psi_1\|_{L^{\infty}}\leq C.
\end{align}
Moreover, for any $k=1,\cdots,N$, $\pa^k_y z_1$ is continuous and $\pa^k_y\psi_1$ is continuous away from $\mathbb{R}\times \{v_1=s\}$.
\end{corollary}
\begin{proof}
Utilizing Corollary \ref{cor3.2.5} and the same argument as in Proposition \ref{prop3.3}, we can show that
\begin{align}
\sum_{k=0}^N|e^{\varpi|\cdot|^{\theta}}\pa^k_yz_1|_{L^{2}_y}+|\pa^k_yz_1|_{L^{\infty}_y}
+\|e^{\varpi|\cdot|^{\theta}}\nu^{1/2}\pa^k_y\psi_1\|_{L^2}+\|w_{q,\vartheta}\pa^k_y\psi_1\|_{L^{\infty}}\leq C,
\nonumber
\end{align}
with some positive constants $C$ independent of $\vep$. The continuity is also similar. Then \eqref{3.3.21} follows from the standard Sobolev embedding.
\end{proof}

\subsection{Proof of the main theorem}
\begin{proof}[Proof of Theorem \ref{thm1.1}] Recall that $y=\vep x$ and $$
\begin{aligned}\bigg[\f{F_A-M_{A,-}}{\sqrt{M_{A,-}}},\f{F_B-M_{B,-}}{\sqrt{M_{B,-}}}\bigg]^T=\vep\mathbf{f}(\vep x)
=\vep \big[(z_0+\vep z_1)\phi_{\vep}+\vep(\psi_0+\vep^{1/2}\psi_1)\big].
\end{aligned}$$
Existence, continuity and \eqref{1.0.2} follow from Proposition \ref{prop3.3}. For $x>0$, we have
$$\begin{aligned}
\bigg[\f{F_A-M_{A,+}}{\sqrt{M_{A,-}}},\f{F_B-M_{B,+}}{\sqrt{M_{B,-}}}\bigg]^T&=\vep[\mathbf{f}-\mathbf{f}_{\infty}](\vep x)\\
&=\vep [p(\vep x)-1][z_{\infty}\phi_{\vep}+\vep(1+ p)\psi_{\infty}]+\vep^2(z_1\phi_{\vep}+\vep^{1/2}\psi_1).
\end{aligned}
$$
Then, \eqref{1.0.3-1} with $x>0$ follows directly from the exponential decay \eqref{p} of $p(\vep x)$ and (sub)exponential decay of $z_1$ and $\psi_1$ in Corollary \ref{cor3.4}. The decay for $x<0$ in \eqref{1.0.3} is similar. Therefore, the proof of Theorem \ref{thm1.1} is completed.
\end{proof}

\section{Appendix}\label{sec.app}
\begin{proof}[Proof of Proposition \ref{prop2.1}]
	 We seek for the solution to \eqref{3.1.1} in the form of
\begin{align}\label{5.1}
\phi_{\vep}=\phi_0+\vep\phi_1+\vep^2\Xi_{\vep}.
\end{align}
Substituting \eqref{5.1} into \eqref{3.1.1} and equating the coefficients on the both sides in front of the zero and first power of $\vep$ respectively, we have
\begin{align}\label{5.2}
\mathbf{L}\phi_0=0,
\end{align}
and
\begin{align}\label{5.3}
\mathbf{L}\phi_1+\lambda(v_1-c_0)\phi_0=0.
\end{align}
From \eqref{5.2} and \eqref{5.3}, $\phi_0=\sum_{j=-1}^4\alpha_j\chi_j$ where $\alpha_j$ is determined by the following linear algebraic system
\begin{align}\label{5.4}
\langle(v_1-c_0)\phi_0,\chi_i\rangle=0,\quad i=-1,\cdots,4.
\end{align}
Here, $\chi_i\in\text{Ker}\, \mathbf{L}, i=-1,\cdots4$ are defined in \eqref{2.1.2}. A direct computation shows that the coefficient matrix of \eqref{5.4} is given by
$$
\{\langle(v_1-c_0)\chi_i,\chi_j\rangle\}_{i,j=-1}^4=
\left(
  \begin{array}{cccccc}
    -c_0\rho_{A,-}& 0 & \rho_{A,-} & 0& 0 & 0 \\
    0& -c_0\rho_{B,-} & \rho_{B,-} & 0 & 0 & 0\\
     \rho_{A,-}& \rho_{B,-} & -c_0\b{m}_- & 0 & 0 & \b{\rho}_- \\
    0 & 0& 0 & -c_0\b{m}_- & 0 & 0 \\
    0 & 0 & 0 & 0 & -c_0\b{m}_- & 0 \\
    0 & 0 & \b{\rho}_- & 0 & 0 & -\frac{3}{2}c_0\b{\rho}_-\\
  \end{array}
\right),
$$
where $\b{\rho}_-$ and $\b{m}_-$ are defined in \eqref{1.2.5}. Therefore, from \eqref{5.4} we solve
\begin{align}
\phi_0=\b{\alpha}(\chi_{-1}+\chi_0+c_0\chi_1+\f23\chi_4):=\b{\alpha}\phi_0',\nonumber
\end{align}
with constant $\b{\alpha}$ to be chosen later. Next, we solve $\phi_1$. Decompose
\begin{align}\label{5.4-1}
\phi_1=\b{\alpha}(\beta'\phi_0'+\sum_{j=-1}^{3}\beta_j\chi_j+\phi_1'):=\b{\alpha}(\b{\phi}_1+\phi_1'),
\end{align}
where $\langle\phi_1',\chi_i\rangle=0$, $i=-1,\cdots,4$. From \eqref{5.3}, $\phi_1'$ is uniquely determined by the equation $L\phi_1'=-\lambda(v_1-c_0)\phi_0'$. To further solve $\beta_j$, $\beta'$ and $\phi_1'$, we shall use the orthogonal conditions \eqref{3.1.2}. Substitute \eqref{5.1} to \eqref{3.1.2} and balance the $O(1)$ term to get
\begin{align}\label{Pro1}
\langle(v_1-c_0)\phi_0,\phi_0\rangle=0.
\end{align}
And balance the $O(\vep)$ term to get
\begin{equation}\label{5.5}
\left\{\begin{aligned}
&2\langle(v_1-c_0)\phi_0,\phi_1\rangle+\langle\phi_0,\phi_0\rangle=-1,\\
&\langle(v_1-c_0)\phi_1,\chi_i\rangle+\langle\chi_i,\phi_0\rangle=0,\quad i=-1,\cdots,4.
\end{aligned}\right.
\end{equation}
The remainder equations of $\Xi_{\vep}$ are
\begin{equation}\label{5.6}
\left\{\begin{aligned}
&\langle(v_1-s)\phi_1,\phi_1\rangle+2\langle\phi_0,\phi_1\rangle+\langle(v_1-s)(2\phi_0+2\vep\phi_1+\vep^2\Xi_{\vep}),\Xi_{\vep}\rangle=0,\\
&\langle(v_1-s)\Xi_{\vep}, \chi_i\rangle+\langle\phi_1,\chi_i\rangle=0, \quad i=-1,\cdots,4.
\end{aligned}\right.
\end{equation}
From $\eqref{5.5}_2$, we have $\langle(v_1-c_0)\phi_1,\phi_0\rangle+\langle\phi_0,\phi_0\rangle=0,$ which together with $\eqref{5.5}_1$, implies that
\begin{align}\label{Pro2}
\langle\phi_0,\phi_0\rangle=-\langle(v_1-c_0)\phi_0,\phi_1\rangle=1.
\end{align}
Therefore, we have
 \begin{align}\notag
\b{\alpha}^2=\f3{10\b{\rho}_-},\quad\lambda=\b{\alpha}^{-2}\langle(v_1-c_0)\phi_0',\mathbf{L}^{-1}(v_1-c_0)\phi_0'\rangle^{-1}>0.
 \end{align}
 Now we turn to solve coefficients $\beta'$, $\beta_j$ in \eqref{5.4-1} as well as $\Xi_{\vep}$. We see from $\eqref{5.5}_2$ that
\begin{align}\label{5.7}
\sum_{j=-1}^3\langle(v_1-c_0)\chi_i,\chi_j\rangle\beta_j=-\langle\chi_i,\phi_0'\rangle-\langle(v_1-c_0)\chi_i,\phi_1'\rangle,
\end{align}
for $i=-1,\cdots 3.$ A direct computation shows that the determinant of coefficient matrix of \eqref{5.7} satisfies
\begin{align}\label{5.7-1}
\det\{\langle(v_1-c_0)\chi_i,\chi_j\rangle\}_{i,j=-1}^3=-\f23c_0^3\b{m}_-^2\b{\rho}_-\rho_{A,-}\rho_{B,-}<0.
\end{align}
Therefore, $\{\beta_j\}_{j=-1}^3$ can be uniquely solved from \eqref{5.7}. It remains to solve $\beta'$ and $\Xi_{\vep}$ under the constraint of \eqref{5.6}. Actually, the choices of $\beta'$ and $\Xi_{\vep}$ are not unique and we only give one among them. Let 
\begin{align}\label{DefXi}
\Xi_{\vep}=\sum_{j=-1}^3\kappa_j\chi_j+\kappa'\chi_5,
\end{align}
 where
$$
\chi_5=[(m_Av_2^2-1)M_A^{1/2},(m_Bv_2^2-1)M_B^{1/2}]^T.
$$
Note that $\langle(v_1-s)\chi_i,\chi_5\rangle=0$, for $i=-1,\cdots,3$. Then, one has from $\eqref{5.6}_2$ that
\begin{align}\label{5.7-2}
\sum_{j=-1}^3\langle(v_1-s)\chi_j,\chi_i\rangle\kappa_j=-\b{\alpha}\langle\beta'\phi_0'+\sum_{j=-1}^{3}\beta_j\chi_j+\phi_1',\chi_i\rangle,
\end{align}
for $i=-1,\cdots,3.$ As for the last two unknowns $\beta'$ and $\kappa'$, by letting $i=4$, it follows from $\eqref{5.6}_2$ that
\begin{align}\label{5.8}
&\langle(v_1-s)\chi_5,\chi_4\rangle\kappa'+\langle\phi_0,\chi_4\rangle\beta'=-\sum_{j=-1}^{3}\langle(v_1-s)\chi_j,\chi_4\rangle\kappa_j,
\end{align}
Notice that by the expansion \eqref{5.4-1} and \eqref{DefXi}, together with properties \eqref{5.4}, \eqref{Pro1} and \eqref{Pro2}, we can rewrite $\eqref{5.6}_1$ in the form of
$$
2\langle(v_1-c_0)\phi_0,\b{\alpha}(\sum_{j=-1}^{3}\beta_j\chi_j+\phi_1')\rangle\beta'+2\langle\phi_0,\phi_0\rangle\beta'+2\langle(v_1-c_0)\phi_0,\kappa'\chi_5\rangle=O(1)+O(\eps).
$$
We further use the identity
\begin{align*}
&\langle(v_1-c_0)\phi_0,\b{\alpha}(\sum_{j=-1}^{3}\beta_j\chi_j+\phi_1')\rangle=\langle(v_1-c_0)\phi_0,\phi_0+\b{\alpha}(\sum_{j=-1}^{3}\beta_j\chi_j+\phi_1')\rangle\\
&=\langle(v_1-c_0)\phi_0,\phi_1\rangle=-\langle(v_1-c_0)\phi_0,\phi_0\rangle=-1
\end{align*}
to get
\begin{align}\label{Eqka'}
\langle(v_1-c_0)\phi_0,\chi_5\rangle\kappa'=O(1)+O(\eps).
\end{align}
Direct calculations show that the coefficients in \eqref{5.8} and \eqref{Eqka'} are given by
$$
\langle(v_1-s)\chi_5,\chi_4\rangle= -s\b{\rho}_-, \qquad\langle\phi_0,\chi_4\rangle = \sqrt{\frac{3\b{\rho}_-}{10}},\qquad \langle(v_1-c_0)\phi_0,\chi_5\rangle = -c_0\sqrt{\frac{2\b{\rho}_-}{15}}.
$$
Moreover, if we put the $O(\eps)$ term on the right hand side, then $(\beta',\ka_{-1},\cdots,\ka_3,\ka')$ forms a linear system with the coefficient matrix
$$
\mathcal M=
\begin{pmatrix}
	-\bar\alpha\rho_{A,-} & -s\rho_{A,-} & 0 & \rho_{A,-} & 0 & 0 & 0\\
	-\bar\alpha\rho_{B,-} & 0 & -s\rho_{B,-} & \rho_{B,-} & 0 & 0 & 0\\
	-\bar\alpha c_0{\bar m}_- & \rho_{A,-} & \rho_{B,-} & -s{\bar m}_- & 0 & 0 & 0\\
	0 & 0 & 0 & 0 & -s{\bar m}_- & 0 & 0\\
	0 & 0 & 0 & 0 & 0 & -s{\bar m}_- & 0\\
	\bar\alpha{\bar\rho}_- & 0 & 0 & {\bar\rho}_- & 0 & 0 & -s{\bar\rho}_-\\
	0 & 0 & 0 & 0 & 0 & 0 &
	-c_0\sqrt{\dfrac{2{\bar\rho}_-}{15}}
\end{pmatrix},
$$
and its determinant 
$$
\det \mathcal M
=
\frac{1}{5}\,
c_0\,{\bar m}_-^{2}\rho_{A,-}\rho_{B,-}\,{\bar\rho}_-s^3
\left(2{\bar\rho}_-+{\bar m}_- s\varepsilon\right)>0.
$$
Hence, one can solve $\beta',\ka_{-1},\cdots, \ka_{3}$ and $\kappa'$ from \eqref{5.8}, \eqref{Eqka'} and the iteration argument. Finally, let $\mu_{\vep}:=\mathbf{L}\Xi_{\vep}+\lambda\phi_0+\lambda(v_1-s)(\phi_1+\vep\Xi_{\vep})$, then it is direct to show that $\mu_{\vep}$ satisfies \eqref{3.1.3} and \eqref{3.1.4}. The proof of Proposition \ref{prop2.1} is completed.
\end{proof}

The following elementary lemma is used in decay estimate of shock profile.
\begin{lemma}\label{lmA.1}
(1) If $0<\lambda_1<\lambda_2,$ then
\begin{align}\label{5.9}
\left|\int_0^ye^{-\lambda_1|y-y_1|}e^{-\lambda_2|y_1|}\dd y_1\right|\leq \f1{\lambda_2-\lambda_1}e^{-\lambda_2|y|}.
\end{align}
(2) If $0\leq\theta<1$, then for any $\lambda_1,\lambda_2>0$, it holds that
\begin{align}\label{5.10}
\left|\int_0^ye^{-\lambda_1|y-y_1|}e^{-\lambda_2|y_1|^{\theta}}\dd y_1\right|\leq C_{\lambda_1,\lambda_2,\theta}e^{-\lambda_2|y|^{\theta}},
\end{align}
where the constant $C_{\lambda_1,\lambda_2,\theta}>0$ is independent of $y.$
\end{lemma}
\begin{proof}
We only prove \eqref{5.10}. Without loss of generality, we assume that $y>0$. By Young's inequality, one has
$$
\lambda_2|y-y_1|^{\theta}\leq \f{\lambda_1|y-y_1|}{2}+(1-\theta)\lambda_2^{\f1{1-\theta}}\left(\f{2\theta}{\lambda_1}\right)^{\f{\theta}{1-\theta}}.
$$
Therefore, by using the elementary fact $|y|^\theta\leq |y-y_1|^{\theta}+|y|^{\theta}$, one has
$$
\begin{aligned}
\left|\int_0^ye^{-\lambda_1(y-y_1)}e^{-\lambda_2y_1^{\theta}}\dd y_1\right|&\leq e^{-\lambda_2y^\theta} e^{(1-\theta)\lambda_2^{\f1{1-\theta}}\left(\f{2\theta}{\lambda_1}\right)^{\f{\theta}{1-\theta}}}
\int_0^ye^{-\f{\lambda_1(y-y_1)}{2}}\dd y_1\\
&\leq \f2{\lambda_1}e^{(1-\theta)\lambda_2^{\f1{1-\theta}}\left(\f{2\theta}{\lambda_1}\right)^{\f{\theta}{1-\theta}}} e^{-\lambda_2y^\theta}:=C_{\lambda_1,\lambda_2,\theta}e^{-\lambda_2y^\theta}.
\end{aligned}
$$
\end{proof}

A similar calculation as in \cite{BD} gives the following Carleman's representation of terms in $K_R$.

\begin{lemma}\label{lmA.2}
Let $0<m\leq 1$ and $i,j\in\{A,B\}$. We have
\begin{itemize}
\item Carleman's representation for loss term \eqref{loss}:
\begin{align}\label{c1}
\int_{\mathbb{R}^3}\int_{\mathbb{S}^2}B^{ji}&(|v-u|,\sigma)\chi_m(|v-u|)M_j^{1/2}(u)M_{i}^{1/2}(v)f_j(u)\dd u\dd\sigma\nonumber\\
&=C_{ij}\int_{\mathbb{R}^3}\chi_m(|v-u|)e^{-\f{m_i|v|^2+m_j|u|^2}{4}}|v-u|^\gamma f_j(u)\dd u.
\end{align}
Here constant $C_{ij}>0$ depends only on $m_i$, $m_j$ and cross-section $B^{ji}$.\\
\item Carleman's representation for mono-species part of gain term \eqref{gain}:
\begin{align}
&\int_{\mathbb{R}^3}\int_{\mathbb{S}^2}B^{ii}(|v-u|, \theta)\chi_m(|v-u|)M_i^{1/2}(u)[M_{i}^{1/2}(v')f_i(u')+M_{i}^{1/2}(u')f_i(v')]\dd u\dd\sigma\nonumber\\
&=C_{ii}\int_{\mathbb{R}^3}\f{f_i(v')}{|v-v'|}e^{-\f{m_i}{8}[|v-v'|^2+\f{||v|^2-|v'|^2|^2}{|v-v'|^2}]}\dd v'\int_{\mathbb{R}^2}\f{b^{ii}(\theta)\chi_m(\sqrt{|v-v'|^2+|\eta|^2})}{(|v-v'|^2+|\eta|^2)^{\f{1-\gamma}{2}}}e^{-\f{m_i|\eta+v_{\perp}|^2}{2}}\dd\eta.\notag
\end{align}
Here $v_{\perp}$ is the projection of $v$ on the hyperplane perpendicular to $\f{v-v'}{|v-v'|}$ and passing through $0$ and the constant $C_{ii}>0$ depends only on $m_i$ and cross-section $B^{ii}$.\\
\item Carleman's representation for bi-species part of gain term \eqref{gain}: for $i\neq j$ and $m_i\neq m_j$, it holds that
\begin{align}
&\int_{\mathbb{R}^3}\int_{\mathbb{S}^2}B^{ji}(|v-u|,\sigma)\chi_m(|v-u|))M_j^{1/2}(u)M_{j}^{1/2}(u')f_{i}(v')\dd u\dd\sigma\nonumber\\
=&C_{ij}\int_{\mathbb{R}^3}\f{f_i(v')}{|v-v'|}e^{-\f{m_i^2|v-v'|^2}{8m_j}-\f{m_i||v|^2-|v'|^2|^2}{8|v-v'|^2}}\dd v' \nonumber\\
&\quad\times\int_{\mathbb{R}^2}\f{b^{ji}(\theta)\chi_m\left(\sqrt{|\eta|^2+\f{|v'-v|^2}{4}\left(1+\f{m_i}{m_j}\right)^2}\right)}
{\left(|\eta|^2+\f{|v'-v|^2}{4}\left(1+\f{m_i}{m_j}\right)^2\right)^{\f{1-\gamma}{2}}}e^{-\f{|v_{\perp}+\eta|^2}{2}}\dd\eta, \label{c3}
\end{align}
and
\begin{align}
&\int_{\mathbb{R}^3}\int_{\mathbb{S}^2}B^{ji}(|v-u|,\sigma)\chi_m(|v-u|))M_j^{1/2}(u)M_{i}^{1/2}(v')f_{j}(u')\dd u\dd\sigma\nonumber\\
=&C_{ij}\int_{\mathbb{R}^3}\f{f_j(u')}{|u'-v|}e^{-\left|\f{\sqrt{m_i}+\sqrt{m_j}}{\sqrt{m_i}-\sqrt{m_j}}\right|^2\tilde{V}^T\mathbf{U}\tilde{V}}
e^{-\f{|m_i-m_j|^2|v_{\perp}|^2}{4}}\dd u'\nonumber\\
&\quad\times\int_{|\eta|=\f{|v-u'|}{|m_i-m_j|}}\f{b^{ji}(\theta)\chi_m(|v-u(\eta,v,u')|)}{|v-u(\eta,v,u')|^{1-\gamma}}
e^{-\f{\sqrt{m_im_j}|\sqrt{m_im_j}\eta-z(v,u')|^2}{2}}\dd \eta.\label{c4}
\end{align}
In both \eqref{c3} and \eqref{c4}, we have denoted $v_{\perp}$ as the projection of $v$ on the hyperplane perpendicular to $\f{v-u'}{|v-u'|}$ and passing through $0$ and the constant $C_{ij}>0$ depends only on $m_i$, $m_j$ and cross-section $B^{ji}$. In \eqref{c4}, we have denoted $\tilde{V}:=[|v-u'|,\f{||v|^2-|u'|^2|}{|v-u'|}]^T$, the matrix
\begin{equation}\label{DefU}
\mathbf{U}:=\left(
              \begin{array}{cc}
                \sqrt{\f{|m_i+m_j|^2}{4}+m_im_j}, & \f{m_j^2-m_i^2}{2} \\
                \f{m_j^2-m_i^2}{2},& \f{|m_i-m_j|}{2} \\
              \end{array}
            \right),
\end{equation}
$z=z(v,u')=\f{m_ju'-m_iv}{m_j-m_i}$, and $u=u(\eta,v,u')=z(v,u')+m_i\eta$.
\end{itemize}
\end{lemma}

\bigskip 
\noindent{\bf Acknowledgements.} 
The research of Renjun Duan was partially supported by the General Research Fund (Project No.~14303523) from  RGC of Hong Kong and by the National Natural Science Foundation of China   (Project No.~12425109). Zongguang Li would like to thank the Research Centre for Nonlinear Analysis at The Hong Kong Polytechnic University for supporting his postdoc study. The research of Zhu Zhang was supported by the Early Career Scheme (Project No.~25303523).

\vspace{2mm}
\noindent\textbf{Conflict of interest.} The authors do not have any possible conflicts of interest.

\vspace{2mm}
\noindent\textbf{Data availability statement.}
 Data sharing is not applicable to this article as no data sets were generated or analyzed during the current study.

\end{document}